\documentclass[journal]{IEEEtran}
\usepackage{cite}
\ifCLASSINFOpdf
   \usepackage[pdftex]{graphicx}
\else
   \usepackage[dvips]{graphicx}
\fi

\usepackage{mathrsfs,amsfonts,amssymb,amsthm}

\usepackage[cmex10,intlimits]{amsmath}

\usepackage{array}

\usepackage{mdwmath}
\usepackage{mdwtab}

\usepackage[font=footnotesize]{subfig}
\usepackage{color}

\begin{document}

\newtheorem{theorem}{Theorem}[section]
\newtheorem{proposition}{Proposition}[section]
\newtheorem{corollary}{Corollary}[section]
\newtheorem{lemma}{Lemma}[section]
\newtheorem{definition}{Definition}[section]
\newtheorem{assumption}{Assumption}
\newcommand{\ams}{\textit{Ann. Math. Statist.}}
\newcommand{\dfn}{\stackrel{\triangle}{=}}
\newcommand{\argmax}{\operatornamewithlimits{arg\,max}}
\newcommand{\argmin}{\operatornamewithlimits{arg\,min}}
\newcommand{\argsup}{\operatornamewithlimits{arg\,sup}}
\newcommand{\arginf}{\operatornamewithlimits{arg\,inf}}

\renewcommand{\thetheorem}{\arabic{section}.\arabic{theorem}}
\renewcommand{\thelemma}{\arabic{section}.\arabic{lemma}}
\renewcommand{\thedefinition}{\arabic{section}.\arabic{definition}}
\renewcommand{\theproposition}{\arabic{section}.\arabic{proposition}}
\renewcommand{\thecorollary}{\arabic{section}.\arabic{corollary}}


\title{Asymptotic Optimality Theory For Decentralized Sequential Multihypothesis Testing Problems}

\author{Yan Wang,
Yajun Mei
\thanks{This work was supported in part by the AFOSR grant FA9550-08-1-0376 and the NSF Grant CCF-0830472. The material in this paper was presented in part at the IEEE International Symposium on Information Theory, Austin, TA, June 2010.}

\thanks{Y. Wang and Y. Mei are with the School of Industrial and System Engineering, Georgia Institute of Technology, Atlanta, GA, 30332 USA (e-mail: ywang67@isye.gatech.edu, ymei@isye.gatech.edu).}}

\maketitle

\begin{abstract}
\boldmath
\bf The Bayesian formulation of sequentially testing $M \ge 3$ hypotheses is studied in the context of a decentralized sensor network system. In such a system, local sensors observe raw observations and send quantized  sensor messages to a fusion center which makes a final decision when stopping taking observations. Asymptotically optimal decentralized sequential tests are developed from a class of ``two-stage" tests that allows the sensor network system to make a preliminary decision in the first stage and then optimize each local sensor quantizer accordingly in the second stage. It is shown that the optimal local quantizer at each local sensor in the second stage can be defined as a maximin quantizer which turns out to be a randomization of at most $M-1$ unambiguous likelihood quantizers (ULQ). We first present in detail our results for the system with a single sensor and binary sensor messages, and then extend to more general cases involving any finite alphabet sensor messages, multiple sensors, or composite hypotheses.
\end{abstract}

\IEEEpeerreviewmaketitle

\section{Introduction}\label{sec:intro}

Sequential detection or sequential hypothesis testing has many important real-world applications such as target detection in multiple-resolution radar (Marcus and Swerling \cite{ms}), serial acquisition of direct-sequence spread spectrum signals (Simon et al. \cite{sosl}) and statistical pattern recognition (Fu \cite{fu}).
The centralized version, in which all observations are available at a single central location, has been well studied.
For example, when testing $M=2$ hypotheses, a well-known optimal centralized test is the sequential probability ratio test (SPRT) developed
by Wald \cite{wald47}, also see Wald and Wolfowitz \cite{ww}. When testing $M\ge 3$ hypotheses, i.e., in the sequential multihypothesis testing problem, there is no tractable closed-form expression for the optimal centralized sequential tests, although various asymptotically optimal sequential tests have been proposed and investigated
in the literature, see, for example, Kiefer and Sacks \cite{ks}, Lorden \cite{lord}, Draglin, Tartakovsky and Veeravalli \cite{dtv, dtv2}.

In recent years, the decentralized version of sequential hypothesis testing problems has gained a great amount of attention and has been applied into a wide range of applications such as military surveillance (Tenney and Sandell \cite{ts}), target tracking and classification (Li et al. \cite{lwhs}), and data filtering (Ye et al. \cite{yllz}). Under a widely used decentralized setting, raw data are observed at a set of geographically deployed sensors, whereas the final decision is made at a central location, often called the fusion center. The key feature here is that raw observations at the local sensors are generally not directly accessible by the fusion center, and the local sensors need to send quantized summary messages (generally belonging to a finite alphabet set) to the fusion center. This is due to limited communication bandwidth and requirements of high communication robustness.

Unfortunately, decentralized sequential hypothesis testing problems are very challenging, and to the best of our knowledge, existing research is restricted to testing two simple hypotheses, for example, see Veeravalli \cite{vee}, Veeravalli, Basar and Poor \cite{vbp},  Nguyen, Wainwright and Jordan \cite{nw}, and Mei \cite{mei}. It has been an open problem to find any sort of asymptotically optimal solutions for the decentralized sequential testing problem when testing $M \ge 3$ hypotheses. This is not surprising, because  even in the centralized version, it requires sophisticated mathematical and statistical techniques and only asymptotic optimality results are available.

The primary goal of this paper is to develop a class of asymptotically optimal decentralized sequential procedures for testing $M \ge 3$ hypotheses. To do so, a major challenge we need to overcome is finding the ``optimal quantizers" that can best send quantized summary sensor messages from the local sensors to the fusion center so as to lose as little information as possible. Intuitively, such a quantizer should depend on the true distribution of the raw data, which is unknown, and thus stationary quantizers are generally not optimal. In addition, since a quantizer can be any measurable function as long as its range is in the given finite alphabet set, it resides in an infinite dimensional functional space. Hence it is essential to investigate the form of  the ``optimal quantizers" so that one can reduce the infinite dimensional functional space to a finite-dimensional parameter space for the purpose of theoretical analysis and numerical computation. Note that when testing $M=2$ hypotheses, Tsitsiklis \cite{tsi} and Veeravalli et al.  \cite{vbp} showed that the optimal quantizers can be found from the family of monotone likelihood ratio quantizers (MLRQ), whose form is defined up to a finite number of parameters. Unfortunately, such a result does not apply to the case of testing $M \ge 3$ hypotheses. To find the form of the optimal quantizers for multi-hypotheses, we propose to
combine three existing methodologies together: two-stage tests in Stein \cite{stein} and  Kiefer and Sacks \cite{ks} (or equivalently, tandem quantizers in Mei \cite{mei}), unambiguous likelihood quantizers (ULQ) in Tsitsiklis \cite{tsi}, and randomized quantizers (see Chernoff \cite{che} for a closely related topic on randomized experiments).


The remainder of the paper is organized as follows. Section \ref{sec:pbfm} gives a rigorous formulation of decentralized sequential multihypothesis testing  problems under a Bayesian framework. Section \ref{sec:2stage} provides a general definition of two-stage tests and discusses their implementation issues, especially those of the randomized quantizers. To highlight our main ideas, Section \ref{sec:main} states our main results for a simplified sensor network system with a single sensor and binary sensor messages: Subsection \ref{subs:rdqntzInfoNmbr} develops asymptotically optimal decentralized sequential tests by considering two-stage tests when the local quantizers are the proposed ``maximin quantizers,'' and Subsection \ref{subs:SearchMaximinQntzr} characterizes maximin quantizers and discusses their numerical computation issues. Section \ref{sec:ext} extends our main results to three more general cases: (A) systems with finite alphabet sensor messages; (B) systems with conditionally independent multiple sensors; and (C) testing composite hypotheses. Numerical simulation results  are presented in Section \ref{sec:Exp}, and concluding remarks are included in Section \ref{sec:conclusion}. The  technical details are provided in the appendices.

\section{Notation and Problem Formulation}\label{sec:pbfm}

\begin{figure}
\centering
\includegraphics[width=3in]{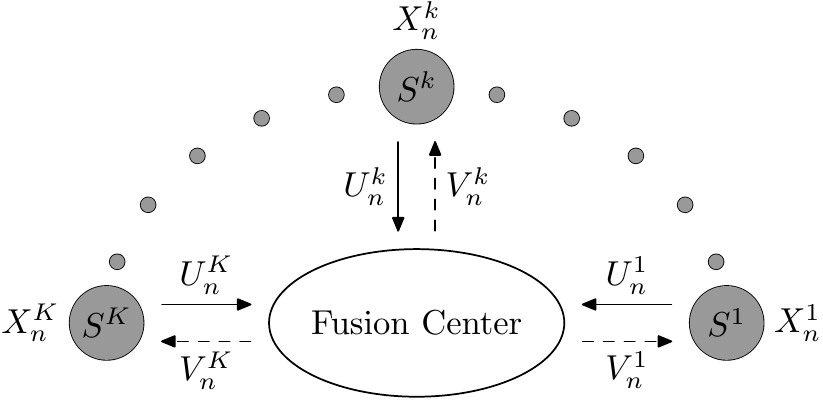}
\caption{A widely used configuration of sensor network}
\label{fig:snsrnet}
\end{figure}

As illustrated in Fig.\ref{fig:snsrnet}, in a widely used configuration, a sensor network consists of  $K$ local sensors labeled by $S^1$, \dots, $S^K$ and a fusion center which makes a final decision when stopping taking observations.
At each time step $n=1,2,\ldots,$ each local sensor $S^k$ observes raw data $\{X^k_n\}$ and sends quantized summary messages $\{U^k_n\}$ to the fusion center. Here the quantized messages $\{U^k_n\}$ are required to belong to a finite alphabet, say, $\{0, 1, \dots, l^k-1\},$ due to limited communication bandwidth or requirements of high communication robustness. In other words, the fusion center does not have direct access to the raw data $\{X^k_n\}$, and have to utilize the quantized sensor messages $\{U^k_n\}$ to make a final decision. If necessary, the fusion center can send feedback $\{V_n^k\}$ to the local sensors so as to improve the system efficiency.

To be more rigorous, we need to further specify the form of the sensor message functions. In this paper, we focus on \textit{systems with full feedback, but local memories restricted to past decisions}, e.g., \textit{Case E} of Veeravalli et al. \cite{vbp}. Mathematically, at time $n,$ for each $k=1,2,\ldots, K,$ the quantized sensor message at the $k$th local sensor is assumed to be of the form
\begin{equation}\label{equ:UinPhiinXinVinminus1}
U^k_n=\phi^k_n(X^k_n; V^{k}_{n-1}) \ \in\{0,1,\dots, l^k-1\}
\end{equation}
where the feedback $V^k_{n-1}$ is defined by
\begin{equation}\label{equ:VinpsiinUV}
  V^k_{n-1}=\psi^k_n(U^{1}_{[1,n-1]},\dots,U^{K}_{[1,n-1]})
\end{equation}
and $U^k_{[1,n-1]}=(U^k_1,\dots, U^k_{n-1})$ denotes all past local sensor messages.  That is, the quantizer $\phi_n^k$ is a function used by sensor $S^k$ to map the local raw data $X^k_n$ into  $\{0,1,\dots, l^k-1\}$, and the choice of $\phi_n^k$ can depend on the feedback $V^k_{n-1}$ and can be a randomized function (to be discussed later).

In decentralized sequential multihypothesis testing problems, there are $M$ hypotheses regarding the distribution $\textbf{P}$ of the raw data $\{X^k_n\}$:
\begin{equation} \label{equ:Mhypo}
  \textbf{H}_m: \quad \textbf{P}=\textbf{P}_m, \quad m=0,1,\dots,M-1.
\end{equation}
Under each $\textbf{P}_m,$ the raw data $X^k_n$ at local sensor $S^{k}$ are i.i.d. with density $f_m^k(\cdot)$ with respect to a common underlying measure, and the raw data $\{X^k_n\}$ are assumed to be independent among different sensors. Hence the distributions of the raw data under $\textbf{P}_m$ are  completely determined by the $K$ densities: $f_m^1$,\dots, $f_m^K.$ Below we simply state that the true state of nature is $m$ or $\textbf{P}_m$ if the hypothesis $\textbf{H}_m$ is true.

A decentralized sequential test $\delta$ consists of a rule to determine the sensor messages, a stopping time $N$ used by the fusion center and a final decision rule $D \in\{0, 1,\dots, M-1\}$ that chooses one of the $M$ probability measures $\textbf{P}_{m}$'s based on the information up to time $N$ at the fusion center. As in Wald \cite{wald47}, Veeravalli et al. \cite{vbp}, and Veeravalli \cite{vee}, let $c>0$ be the cost per time step until stopping, and let $W(m,m')$ be the loss of making decision $D=m'$ when the true state is $\textbf{P}_m$. It is standard to assume that $W(m,m)=0$ but $W(m,m')>0$ for any $m\neq m'$, i.e., no loss occurs if and only if a correct decision is made. Then when the true state of nature is $\textbf{P}_{m},$ the total expected cost of a decentralized test $\delta$ is
\begin{displaymath}
  \mathcal{R}_c(\delta; m)=c\textbf{E}_m (N) +\sum_{m'}W(m,m')\textbf{P}_m\{D=m'\}
\end{displaymath}
where $\textbf{E}_m$ is the expectation operator under $\textbf{P}_m$. In a Bayesian formulation, we assign prior probabilities $\pi=(\pi_{0},\dots,\pi_{M-1})$ to the $M$ hypotheses $\textbf{H}_0,\cdots,\textbf{H}_{M-1}.$ Hence, the Bayes risk of the decentralized test $\delta$ is
\begin{equation} \label{equ:BysRskSeqTest}
  \mathcal{R}_c(\delta)=\sum \pi_m \mathcal{R}_c(\delta; m).
\end{equation}
The Bayes formulation of the decentralized sequential multihypothesis testing problem can then be stated as follows. \smallskip

\noindent \textit{Problem (P1):} Minimize the $\mathcal{R}_c(\delta)$ in (\ref{equ:BysRskSeqTest}) among all possible decentralized sequential multihypothesis test procedures $\delta$.\smallskip

Denote by $\delta_B^*(c)$ a Bayes solution to (\textit{P1}). In Veeravalli et al. \cite{vbp}, $\delta_B^*(c)$ is obtained through dynamic programming for the simplest case of testing binary hypotheses, i.e., $M=2$. Unfortunately, in a general multihypothesis setting, when $M\ge 3$, it is impossible to find $\delta_B^*(c),$ since the problem is intractable even for the centralized version, see, for example, Dragalin, Tartakovsky and Veeravalli \cite{dtv}. This prompts us to adopt the following asymptotic optimization approach in which the cost $c$ per time step goes to $0$. \smallskip

\noindent\textit{Problem (P2)}: Find a family of decentralized sequential multihypothesis testing procedures $\{\delta_A(c)\}$ that is {\it asymptotically optimal} in the sense that
\begin{equation}\label{equ:AsmptOptmlTest}
\lim_{c\to 0}\mathcal{R}_c(\delta_B^*(c))/\mathcal{R}_c(\delta_A(c))=1.
\end{equation}

Now let us discuss the concepts of quantizers and their Kullback-Leibler (K-L) divergences, both of which will be essential in our asymptotic optimality theory. A quantizer is either a deterministic measurable function or a randomization of some (possibly infinitely many) deterministic  measurable functions that
maps the raw data into a finite alphabet set, e.g., the function $\phi_n^k$ in (\ref{equ:UinPhiinXinVinminus1}) is a quantizer. The quantizer is called a deterministic quantizer if the corresponding measurable function is deterministic. At a given local sensor $S$ (here and below we miss the superscript $k$ for simplicity), denote by $\Phi$ the set of all possible local deterministic quantizers $\phi$'s  and let $f_m(\cdot; \phi)$ be the induced probability mass function of quantized message $U_{n} = \phi(X_n)$ when the raw observation $X_{n}$ is distributed according to $f_m(\cdot)$ under $\textbf{P}_m,$ i.e.,
\begin{equation} \label{equ:qntzrdistri}
  f_{m}(u;{\phi})=\textbf{P}_m({\phi}(X_n)=u), \quad \mbox{ for } u=0,1,\dots, l-1.
\end{equation}
For the deterministic quantizer $\phi,$ it is easy to see that its K-L divergences are defined by
\begin{equation} \label{equ:infodeterm}
I(m,m'; \phi)=\sum_{u=0}^{l-1}f_{m}(u; \phi)\log \frac{f_{m}(u; \phi)}{f_{m'}(u; \phi)}
\end{equation}
for all $m \ne m'.$ However, we need to be very careful when defining the K-L divergences of a randomized quantizer $\bar{\phi}=\sum p^j \phi^j$ that assigns probability masses $\{p^j\}$ onto some countable subset of deterministic quantizers $\{\phi^j\}\subset\Phi$.
On the one hand, one can  directly substituting the $\phi$ in (\ref{equ:infodeterm}) by $\bar{\phi}$, i.e.,
\begin{equation} \label{equ:inforndma}
 \tilde I(m,m'; \bar{\phi})=\sum_{u=0}^{l-1}f_{m}(u;\bar{\phi})\log \frac{f_{m}(u;\bar{\phi})}{f_{m'}(u;\bar{\phi})}
\end{equation}
where
\begin{equation*}
  f_{m}(u; \bar{\phi})=\textbf{P}_m(\bar{\phi}(X)=u), \quad u=0,1,\dots, l-1.
\end{equation*}
This type of the K-L divergence has been defined  for randomized quantizers in the engineering literature, e.g., Tsitsiklis \cite{tsi}. On the
other hand, one can also define the K-L divergence of the randomized quantizer $\bar \phi$ by the weighted average of those of the deterministic quantizers it randomizes:
\begin{equation} \label{equ:inforndm}
I(m,m'; {\bar{\phi}})=\sum p^j I(m,m';{\phi^j}), \quad 0\le m\neq m'\le M-1.
\end{equation}
By Jensen's inequality, we have  $ \tilde I(m,m'; \bar{\phi}) \le I(m,m'; {\bar{\phi}}),$ i.e., the K-L divergence defined in (\ref{equ:inforndma}) is dominated by that in (\ref{equ:inforndm}), also see Appendix \ref{app:qntzr} for more discussions.

To the best of our knowledge,  the K-L divergence in (\ref{equ:inforndm}) has not been studied in the literature so far, and it turns out that it will play a central role in our asymptotic theory. The reason why our asymptotic theory involves the K-L divergence in (\ref{equ:inforndm}) instead of that in (\ref{equ:inforndma}) is due to our novel way of implementing randomized quantizers to minimize loss of information. Roughly speaking, when implementing randomized quantizers, it is essential for the fusion center to know  which specific deterministic quantizer is going to be used at the local sensor at each time step,  since otherwise the fusion center can be confused by randomized quantizers  and its decision making will be less efficient. This issue will be discussed further in Subsection \ref{subs:rndmqntzr}. Also note that a deterministic quantizer can also be thought as a randomized quantizer that assigns probability one to itself. Denote by $\bar{\Phi}$ the set of all possible quantizers at the local sensor $S$, deterministic or randomized.

Throughout our paper we make the following standard assumption to ensure the finiteness of the expectation of the raw data's log-likelihood ratios.
\begin{assumption}\label{ass:KLInfoLimit}
For any two different states $0\le m\neq m'\le M-1$ and local sensor $S^k$,
\[0<\textbf{E}_m\left[\log\frac{f^k_m({X_n^k})}{f^k_{m'}({X_n^k})}\right]<\infty.\]
\end{assumption}
In the literature, researchers often assume a uniform bound on the second moments of the log-likelihood ratio $\log\frac{f_m^k(u^k; \phi)}{f_{m'}^k(u^k; \phi)}$ under $\textbf{P}_m$. See, for example, Kiefer and Sacks \cite{ks} and Mei \cite{mei}. Here our assumption is much weaker, and it turns out that it will be sufficient for the first-order asymptotic optimality under our setting.

\section{Two-Stage Test Procedures} \label{sec:2stage}

In this section, we introduce a class of ``two-stage'' decentralized sequential tests in which each local sensor uses two stationary (possibly randomized) local quantizers with at most one switch between these two quantizers. This type of tests are useful because they allows the fusion center to first make a preliminary guess about the true state of nature and then optimize the procedure accordingly.

To highlight our main ideas, in the present and next sections we assume that the sensor network system consists of a single local sensor, i.e., $K=1$ and all quantized messages are binary, i.e., $U_n\in\{0,1\}$. Extensions to general cases are presented in Section \ref{sec:ext}. To save notations, we drop all the superscripts denoting the sensors. That is, in this and next sections we assume that one observes raw data $X_{1}, X_{2}, \cdots,$ which are i.i.d. with density $f_{m}(x)$ under the hypothesis $\textbf{H}_m.$ The final decision is based on quantized messages $U_{n} = \phi_{n}(X_{n}; V_{n-1}) \in \{0,1\}$ with the feedback $V_{n-1} = \psi_{n-1}(U_{1}, \cdots, U_{n-1}).$ For a given (randomized) quantizer $\phi,$ the K-L divergence of $\textbf{P}_{m'}$ from $\textbf{P}_m$ is $I(m,m'; \phi)$  defined in (\ref{equ:inforndm}).

\subsection{Our Proposed Test} \label{subs:2stage}

 Our proposed two-stage test $\delta(c)$ can be defined as follows. In the \textit{first stage} of $\delta(c)$, the local sensor can use any ``reasonable'' stationary deterministic quantizer and the fusion center needs to make a preliminary guess about the true state of nature.  The only requirement is that as the cost $c\to 0,$ the probabilities of making incorrect preliminary guess go to zero but the time steps taken at this first stage become negligible as compared to those of the overall procedure (or the second stage).

To be more concrete, let $u(c)\in (0,1/2)$ be a function of $c$ such that $u(c)\to 0$ and $\log u(c)/\log c \to 0$ when $c\to 0$, e.g., $u(c)=1/|\log c|$. Choose a deterministic quantizer $\phi^0$ such that $I(m,m';{\phi^0})>0$ for any two states $0\le m\neq m'\le M-1$, and let the local sensor  use the stationary quantizer $\phi^0$ to send i.i.d. sensor messages $U_n=\phi^0(X_n)$ to the fusion center. Then the fusion center faces a classical sequential detection problem with the  i.i.d. sensor messages $U_n$'s as inputs, and thus it is intuitively appealing to make a preliminary decision based on posterior distributions. Specifically, at each time step $n=0, 1, \cdots,$ the fusion center updates recursively the posterior distribution $(\pi_{0,n}, \pi_{1,n},\dots, \pi_{M-1, n})$ as follows:
\begin{equation} \label{equ:postupdate}
  \pi_{m,n}=\frac{\pi_{m,n-1}f_{m}(U_n; \phi^0)}{\sum_{0\le m'\le M-1}\pi_{m',n-1}f_{m'}(U_n;\phi^0)}.
\end{equation}
Then the fusion center will stop the first stage at time step
\begin{equation*}\label{equ:EndOfStage1}
N_0=\min\{n \ge 0: \max_{0\le m\le M-1}\{\pi_{m,n}\}\ge 1-u(c)\}
\end{equation*}
and when stopped, the fusion center  makes a preliminary decision
\begin{equation*}\label{equ:dfnd0predcsn}
D_0=\argmax_{0\le m\le M-1}\pi_{m,N_0}.
\end{equation*}
Note that the preliminary decision $D_0$ is well-defined because the maximum value of $\pi_{m,N_0}$ is attained at only one index $m$  due to the definition of $N_0$ and the fact that $u(c)<1/2.$ For the purpose of practical implementation, the preliminary decision $D_0$ can be transmitted to the local sensor through a feedback of $\log_2 M$ bits.

In the \textit{second stage} of our proposed test $\delta(c),$ the local sensor will switch to another stationary (likely randomized) quantizer  that may depend on  the preliminary decision  $D_0$. Without loss of generality, we assume that the local sensor uses the stationary quantizer $\bar{\phi}_m$ when the preliminary decision at the first stage is $D_0=m$ for $m=0,1,\dots, M-1.$ Here we put a bar over $\bar{\phi}_m$ to emphasize that it is likely a randomized quantizer when optimized, and we will postpone the detailed discussion about how to implement randomized quantizers to the next subsection.

Now at the second stage, the fusion center shall ignore the preliminary decision $D_0$ and {\it continue to update} the posterior distribution $(\pi_{0,n},\dots,\pi_{M-1,n})$ with the sensor messages generated from the new quantizer $\bar{\phi}_{m}$ when $D_0 = m$ (how to update will be discussed in the next subsection).  Then the fusion center will stop the second stage (hence the whole procedure) at time step
\begin{equation}\label{equ:timeToStop2Stagea}
N =\min\{n \ge N_0: \max_{0\le m\le M-1}\{\pi_{m,n}\}\ge 1- c \}
\end{equation}
and when stopped, the fusion center makes a final decision
\begin{equation*}
D =\argmax_{0\le m\le M-1}\pi_{m,N}.
\end{equation*}

From the asymptotic point of view, many other possible decision rules can also be used  at the fusion center. For instance, let $r_{m,n}=\sum_{m'\neq m} \pi_{m',n} W(m',m)$ be the average posterior cost when making a decision $m$ at time $n$, and then the fusion center can stop the second stage at time \begin{equation}\label{equ:timeToStop2Stageb}
N=\min\{N_m: 0\le m\le M-1\}
\end{equation} where
\begin{equation}\label{equ:timeToStop2Stage}
  N_m=\min\{n\ge N_0: r_{m,n}\le c\}, \quad m=0,1,\cdots, M-1.
\end{equation}
Based on our experiences, the stopping time $N$ defined in (\ref{equ:timeToStop2Stageb}) is slightly better than that in (\ref{equ:timeToStop2Stagea}) in finite-sample numerical simulations, especially when the costs $W(m',m)$ are not a simple $0-1$ cost function. Moreover, at the second stage, our proposed test will continue to update the posterior distribution instead of starting afresh as required by the two-stage tests in Section V of Kiefer and Sacks \cite{ks} or in Section IV of Mei \cite{mei}. The main reason is to further utilize information gathered from the first stage so as to improve the efficiency in finite-sample simulations, although it also means extra treatments in asymptotic arguments.

\subsection{Implementing Randomized Quantizers and Updating Posterior Distribution} \label{subs:rndmqntzr}

When testing $M \ge 3$ hypotheses, randomized quantizers are likely needed in the second stage in order to develop the optimal two-stage tests, and thus it is necessary to determine the appropriate approach to implement them as well as how to update posterior distributions at the fusion center, especially at the second stage. Assume a randomized quantizer is given by $\bar{\phi}=\sum p^j \phi^j.$ The key requirements for randomization in our two-stage test is that the fusion center must know which deterministic quantizer is picked to quantize the raw observation, since otherwise the randomization can cause confusion at the fusion center. 
The most straightforward (though practically infeasible) implementation is to let the fusion center do the randomization directly. Specifically, at time step $n$ the fusion center will choose the deterministic quantizer $\phi^j$ with probability $p^j,$ say choosing the deterministic quantizer $\phi^{j(n)}.$ Through a feedback from the fusion center, the local sensor will then use the chosen deterministic quantizer $\phi^{j(n)}$ at time step $n$ to quantize the raw observation. After receiving the quantized sensor message $U_{n}$  at time step $n,$ the fusion center then update the posterior distribution as follows:
\begin{equation}\label{equ:udpostrndmqntzr}
  \pi_{m,n}=\frac{\pi_{m,n-1}f_{m}(U_n; \phi^{j(n)})}{\sum_{m'=0}^{M-1}\pi_{m',n-1}f_{m'}(U_n;\phi^{j(n)})}
\end{equation}
because the fusion center knows that $U_n$ comes from the deterministic quantizer $\phi^{j(n)}$ at time step $n.$

A theoretically equivalent but more feasible implementation in practice is to adopt a ``pseudo-randomization'' at the local level through the so-called ``block design'' (see Section V of Kiefer and Sacks \cite{ks}). To be specific, suppose $\bar{\phi}$ randomizes a finite number (say $i$) of deterministic quantizers, and all $p^j$'s are (or can be approximated by) rational numbers with $b$ a common denominator. Then we divide the time steps into blocks of size $b$, and within each block, the raw data are quantized with deterministic quantizers $\{\phi^1,\dots, \phi^i\}$ following a fixed order such that each $\phi^j$ is used for exactly $p^jb$ times. Under this implementation, the fusion center again knows which deterministic quantizer is used at each time step, and thus can update the posterior distribution as in (\ref{equ:udpostrndmqntzr}).

We would like to point out that our implementation of randomized quantizers is very different from those existing implementations in the literature (see Tsitsiklis \cite{tsi}). In the latter the randomization is done at the local level in the sense that the local sensor randomly picks one of the deterministic quantizer $\phi^j$'s, and the fusion center will only get the quantized message $U_n$ without knowing exactly which deterministic quantizer is used to generate $U_{n}.$ In this case, to update posterior distribution, the fusion center has to plug in $\bar{\phi}$ (instead of $\phi^{j(n)}$) into (\ref{equ:udpostrndmqntzr}), i.e.,
\begin{equation*}
  \pi_{m,n}=\frac{\pi_{m,n-1}f_{m}(U_n; \bar{\phi})}{\sum_{m'=0}^{M-1}\pi_{m',n-1}f_{m'}(U_n; \bar{\phi})}.
\end{equation*}
Since our proposed implementation and the local randomization implementation  lead different likelihood ratios, it is not surprising that there are two different kinds of K-L divergences for a randomized quantizer  in Section \ref{sec:pbfm}: one defined in (\ref{equ:inforndma}) and the other in (\ref{equ:inforndm}).

\section{Main Results}\label{sec:main}
In the present section, we show that a two-stage test can be an asymptotic optimal solution to problem (P2) by carefully choosing the quantizers used in the second stage. We also give characterizations of these optimal quantizers as well as the corresponding numerical computation.

\subsection{Maximin Quantizers and Asymptotic Theory}\label{subs:rdqntzInfoNmbr}

Let us begin with the definition of some useful information numbers. For a given (deterministic or randomized) quantizer $\bar{\phi}\in\bar{\Phi},$ define
\begin{equation} \label{equ:Imbarphimpmm}
I(m; \bar{\phi})=\min_{m'\neq m} I(m, m' ; \bar{\phi}).
\end{equation}
for each state $m=0,1,\dots, M-1.$ That is, $I(m; \bar{\phi})$ characterizes the least divergence from the state $m$ to other states.

The following theorem, whose proof is presented in Appendix \ref{app:theproof},  establishes the asymptotic properties of a two-stage test $\delta(c)$ as the cost $c$ goes to $0$.

\begin{theorem} \label{the:asymp2stageproc}
	Let $\delta(c)$ be a two-stage test with $\{\bar{\phi}_0,\dots, \bar{\phi}_{M-1}\}$ being the set of (possibly randomized) quantizers used in its second stage. Assume each $\bar{\phi}_m$ randomizes a finite number of deterministic quantizers, and suppose that the prior probabilities $\pi_m>0$ and $I(m;{\bar{\phi}_{m'}})>0$ for all states $m=0,1,\dots, M-1$ and $m'=0,1,\dots, M-1$. Then as $c\to 0$, the time steps $N$ taken by the two-stage test $\delta(c)$ satisfies
  \begin{equation} \label{equ:samplesize2stageproc}
	  \textbf{E}_m\left\{ N \right\}=(1+o(1))|\log c|/I(m;{\bar{\phi}_m}), \quad m=0,1,\dots,M-1,
  \end{equation}
  and the final decision $D$ of  the two-stage test $\delta(c)$  satisfies
  \begin{equation} \label{equ:wrongdecision2stageproc}
	  \textbf{P}_m\left\{ D\neq m \right\}=O(c), \quad m=0,1,\dots, M-1.
  \end{equation}
  Thus, the Bayes risk of the two-stage test $\delta(c)$ is
  \begin{equation} \label{equ:risk2stageproc}
	\mathcal{R}_c(\delta)=c|\log c|(1+o(1)) \sum_{m=0}^{M-1} \pi_m/I(m;{\bar{\phi}_m}).
  \end{equation}
\end{theorem}

In light of Theorem \ref{the:asymp2stageproc}, to asymptotically minimize the Bayes risk within the class of two-stage tests, it is clear that one should maximize the information numbers $I(m; \bar{\phi}_m)$ for $m=0,1,\dots, M-1$. This leads to a natural definition of the optimal quantizers that we should use in the second stage:

\begin{definition} \label{def:maximin-info}
For $m=0,1,\dots, M-1$, define the maximin quantizer with respect to $\textbf{P}_m$ as
\begin{equation*}\label{equ:maximins}
  \bar{\phi}^{\textrm{max}}_{m}=\argsup_{\bar{\phi}\in\bar{\Phi}}I(m;{\bar{\phi}})
\end{equation*}
and define the corresponding maximin information number by $I(m)=\sup_{\bar{\phi}\in\bar{\Phi}} I(m;{\bar{\phi}}).$
\end{definition}

As shown later in Theorems \ref{the:MaximinQntzrApproxByULQ} and \ref{the:MaximinULQmultialphabets}, the supremum of $I(m,\bar{\phi})$ is attainable, and the maximin quantizers not only exists, but also can be realized as randomization of a finite number of deterministic quantizers. Now we are ready to investigate the asymptotic optimality properties of the two-stage test when the maximin quantizers are used in the second stage. Denote by $\delta_A(c)$ such a two-stage test. Then by Theorems \ref{the:asymp2stageproc}, we have
 \begin{equation}\label{equ:BayesOptimal}
\mathcal{R}_c(\delta_A(c))=(1+o(1)) c|\log c| \sum_{m=0}^{M-1} \pi_m/I(m).
  \end{equation}
as $c\to 0$. What is surprising is that $\delta_A(c)$ is not only the best one within the class of two-stage tests, but also asymptotically optimal among all possible decentralized tests. A key step in the proof is the following important theorem which establishes asymptotic lower bounds on the expected time steps of any decentralized tests with ``suitably small'' probabilities of making incorrect decisions.

\begin{theorem} \label{the:asympoptdeltaIc}
Assume that $\delta(c)$  is a decentralized (not necessarily a two-stage) test that makes a final decision $D$ and
  \begin{displaymath}
	\textbf{P}_m\{D\neq m\}=O(c\log c),\quad m=0,1,\dots, M-1,
  \end{displaymath}
as $c\to 0.$ Then the time step $N$ taken by $\delta(c)$ satisfies
  \begin{eqnarray} \label{equ:BayesSmplSzs}
	  \textbf{E}_m \left\{ N \right\} &\ge& (|\log c|-\log|\log c|+O(1))/ I(m) \cr
&=& (1+o(1)) |\log c| / I(m)
  \end{eqnarray}
for all $m=0,1,\dots,M-1.$
\end{theorem}
The proof of Theorem \ref{the:asympoptdeltaIc} is presented in Appendix \ref{app:lwbds}.  The first-order asymptotic lower bound will be sufficient to prove the first-order asymptotic optimality of $\delta_A(c),$ and the reason why we present a higher order lower bounds is due to its potential usefulness in higher-order analysis in further research. By relation (\ref{equ:BayesOptimal}) and Theorem \ref{the:asympoptdeltaIc}, we have

\begin{corollary} \label{cor:deltaAoptimal}
The procedure $\delta_A(c)$ is  asymptotically Bayes up to first-order.
\end{corollary}

\begin{proof}  Let $\delta_{B}^{*}(c)$ be the Bayes procedure. By definition,  $\mathcal{R}_c(\delta_{B}^{*}(c)) \le \mathcal{R}_c(\delta_A(c)).$ Using the relation (\ref{equ:BayesOptimal}) and the definition of Bayes risk $\mathcal{R}_c(\delta_{B}^{*}(c)),$
the probabilities for the Bayes procedure $\delta_{B}^{*}(c)$ to make incorrect decisions are at most $O(c\log c).$ By Theorem \ref{the:asympoptdeltaIc}, the stopping time $\tau^{*}_{c}$ of the Bayes procedure $\delta_{B}^{*}(c)$ satisfies (\ref{equ:BayesSmplSzs}). Now using the definition of Bayes risk again, for any test, the cost of time steps taken to make the final decision is only portion of the Bayes risk. In particular,
\[
\mathcal{R}_c(\delta_{B}^{*}(c)) \ge c \sum_{m} \pi_{m} \textbf{E}_m \left\{ \tau^{*}_{c} \right\} \ge  (1+o(1)) c |\log c|  \sum_{m} \pi_{m} / I(m).
\]
Combining all arguments yields that $\mathcal{R}_c(\delta_B^*(c))/\mathcal{R}_c(\delta_A(c)) \rightarrow 1$ as $c \rightarrow 0,$
completing the proof of the corollary.
\end{proof}

It is useful to point out that the  test $\delta_A(c)$  is asymptotic Bayes mainly because the local sensor uses the maximin quantizers $\bar{\phi}^{\textrm{max}}_m$'s in the second stage.  Since the maximin quantizers do not depend on the prior distribution $\{\pi_m\}$'s, it is easy to see from (\ref{equ:samplesize2stageproc}) and (\ref{equ:wrongdecision2stageproc}) that the asymptotic optimality properties  of $\delta_A(c)$ are actually robust with respect to $\{\pi_m\}$ as long as all prior probabilities are positive. Likewise, the asymptotic Bayes properties still hold if the stopping times of $\delta_A(c)$ at the fusion center are replaced by other efficient multi-hypotheses tests, e.g., those in Draglin, Tartakovsky and Veeravalli \cite{dtv, dtv2}.

\subsection{Characterizing the Maximin Quantizers.}\label{subs:SearchMaximinQntzr}
In this subsection, we provide a deeper understanding of the maximin quantizers $\{\bar{\phi}^{\textrm{max}}_m: m=0,1,\dots,M-1\}$ and also illustrate how to compute them explicitly when the sensor messages are binary.

Let us first introduce the unambiguous likelihood quantizer (ULQ) which was first proposed in Tsitsiklis \cite{tsi} as a generalization of Monotone Likelihood Ratio Quantizer (MLRQ). For notational convenience, here we give the definition of ULQ only for the case of binary sensor messages, and the general definition will be provided in Definition \ref{def:ULQGeneral}  in Subsection \ref{subs:non2srmsg}.

\medskip
\begin{definition}\label{def:ULQ}
A deterministic quantizer $\phi\in\Phi$ is said to be an unambiguous likelihood quantizer if there exist real numbers $\{a_m: m=0,\dots,M-1\}$ such that
\begin{equation}\label{equ:dfnULQ1}
\phi(X)=I(\sum_{m=0}^{M-1} a_m f_m(X)>0)
\end{equation}
and for any $0\le m'\le M-1$, the set $\{a_m\}$ satisfies the following condition
\begin{equation}\label{equ:dfnULQ2}
	\textbf{P}_{m'}\left\{  \sum_{m=0}^{M-1} a_m f_m(X)=0\right\}=0.
\end{equation}
\end{definition}

When relation (\ref{equ:dfnULQ2}) holds for any set of $\{a_m\}$ that are not simultaneous zero, the set of pdf's $\{f_m\}$ are said to be \textit{linearly independent}. With the definition of ULQs, the following theorem  characterizes the form of the maximin quantizers $\bar{\phi}^{\textrm{max}}_m.$ The proof is very technical and is deferred to Appendix \ref{app:qntzr}.

\begin{theorem}\label{the:MaximinQntzrApproxByULQ}
	For each  $m=0,1,\dots,M-1$, the maximin quantizer $\bar{\phi}^{\textrm{max}}_m$ exists and can be chosen as a randomization of at most $M-1$ deterministic quantizers.   Moreover, if the pdf's $\{f_m\}$ are linearly independent, then it can actually be chosen as a randomization of at most $M-1$ deterministic ULQ quantizers.
\end{theorem}

Clearly, when testing $M = 2$ simple hypotheses, the ULQs become MLRQs, and thus the maximin quantizers in the second stage is just the deterministic MLRQ, which is consistent with those results in Mei \cite{mei}.

Note that Theorem \ref{the:MaximinQntzrApproxByULQ} reduces the search of the maximin quantizers from an infinite dimensional function space to a parameter space of dimension $O(M^2)$. To see this, fix a state $m$ and define $M^2-1$ parameters as probability masses
$\{p^j_m: 1\le j\le M-1, p^j_m\ge 0, \sum_{j=1}^{M-1} p^j_m=1\},$
and ULQ coefficients
$\{a_{m,m'}^{j}: 1\le j\le M-1, 0\le m'\le M-1, \sum_{m'=0}^{M-1}  (a_{m,m'}^{j})^2=1\}.$
Based on every combination of these parameters, define by $\bar{\phi}$ the quantizer randomizing $M-1$ ULQs:
$\bar{\phi}=\sum_{j=1}^{M-1} p_{m}^j \phi_{m}^j,$
where
\begin{equation*}\label{equ:ULQDcmpdFmMaximinForm}
  \phi_{m}^j(X)=I(\sum_{m'=0}^{M-1} a^j_{m,m'} f_{m'}(X)>0).
\end{equation*}
The maximin quantizer $\bar{\phi}^{\textrm{max}}_m$ can then be found as $\bar{\phi}$ that maximizes
\begin{equation}\label{equ:ObjFunc}
\min_{l \neq m} I(m,l;{\bar{\phi}})
\end{equation}
among all possible combinations of 
\begin{equation*}
\{p^j_m; a^j_{m,m'}\}_{1\le j\le M-1, 0\le m'\le M-1}.
\end{equation*}

To further reduce computational complexity of the maximin quantizers, we can apply the following lemma
which provides a sufficient condition that a deterministic MLRQ quantizer is the maximin quantizer.

\begin{lemma} \label{lem:rdccomplexitesInSolveMmQntzr}
Given $m'\neq m$, let $\phi_{m,m'}$ be the deterministic MLRQ quantizer that maximizes the K-L divergence of $m'$ from $m$, i.e.,
\begin{displaymath}
	\phi_{m,m'}=\argsup_{\bar{\phi}\in\bar{\Phi}} I(m, m';{\bar{\phi}}).
\end{displaymath}
If there exists a state $m'\neq m$ such that for any other state $m''\neq m$:
  \begin{displaymath}
	I(m,m'';{\phi_{m,m'}})\ge I(m,m';{\phi_{m,m'}})
  \end{displaymath}
then $\phi_{m,m'}$ is also the maximin quantizer for state $m$.
\end{lemma}

\begin{proof}
  By definition, \[I(m;{\phi_{m,m'}})=\min_{m''\neq m}I(m,m'';{\phi_{m,m'}})=I(m,m';{\phi_{m,m'}}).\]
  Take any $\bar{\phi}\in\bar{\Phi}$,
  \begin{displaymath}
	I(m;{\bar{\phi}})\le I(m,m';{\bar{\phi}}) \le I(m,m';{\phi_{m,m'}})=I(m;{\phi_{m,m'}})
  \end{displaymath}
 and thus $\phi_{m,m'}$ is the maximin quantizer for state $m$.
\end{proof}

\section{Extensions}\label{sec:ext}

Section \ref{sec:main} deals with the simplest case when the network only has a single sensor with binary sensor messages. In this section, we extend our results to three more general scenarios: 1) the sensor messages  belong to a finite alphabet (not necessarily binary); 2) there is more than one sensor in the network (though observations are independent between different sensors); and 3) the hypotheses are composite.

\subsection{Sensor Messages Belonging to a Finite Alphabet}\label{subs:non2srmsg}

Suppose the network still consists of only one sensor, but now the sensor messages belong to a finite alphabet, say, $\{0,1,\dots, l-1\}$ with $l\ge 2$. In this scenario, the definitions of two-stage tests (Subsection \ref{subs:2stage}) and maximin quantizers (Subsection \ref{subs:rdqntzInfoNmbr})
are still applicable, and Theorem \ref{the:asymp2stageproc} and Theorem \ref{the:asympoptdeltaIc} also hold. The only change is Theorem \ref{the:MaximinQntzrApproxByULQ}, as we need to consider the following general definition of ULQ, which originally proposed in Tsitsiklis \cite{tsi} and includes Definition \ref{def:ULQ} as a special case.

\begin{definition} \label{def:ULQGeneral}
When the sensor messages belong to a finite alphabet $\{0, 1, \dots, l-1\}$, a deterministic quantizer $\phi\in\Phi$ is said to be an unambiguous likelihood quantizer (ULQ) if and only if there exist real numbers $\{a_{i,m}: 0\le i\le l-1, 0\le m\le M-1\}$ such that
  \begin{equation} \label{equ:ULQmultialpha}
	\phi(X)=\argmin_{0 \le i \le l-1} \sum_{m=0}^{M-1} a_{i,m}f_m(X)
  \end{equation}
  and the probability of a tie is zero under every $\textbf{P}_m$ for $m=0,1,\dots, M-1$.
\end{definition}

With this definition, Theorem \ref{the:MaximinQntzrApproxByULQ} can be generalized as follows.
\begin{theorem} \label{the:MaximinULQmultialphabets}
  Suppose the sensor messages belong to a finite alphabet $\{0,1,\dots, l-1\}$ with $l \ge 2$. Then for $m=0,1,\dots, M-1$, the maximin quantizer $\bar{\phi}^{\textrm{max}}_m$ can be realized as a randomization of at most $M-1$ deterministic quantizers. Moreover, for every $m$, there exists a sequence of quantizers $\{\bar{\phi}_{m,i}\}$ each randomizing at most $M-1$ ULQs, such that $I(m;{\bar{\phi}_{m,i}})\to I(m)$, that is, the maximin quantizer $\bar{\phi}^{\textrm{max}}_m$ can be approximated by $\{\bar{\phi}_{m,i}\}$.
\end{theorem}

The  proof of  Theorem \ref{the:MaximinULQmultialphabets} is presented in Appendix \ref{app:qntzr}.
Note that there is a significant difference between Theorem \ref{the:MaximinQntzrApproxByULQ} and Theorem \ref{the:MaximinULQmultialphabets}. When the sensor messages are binary (i.e., $l = 2$), we are sure that the maximin quantizers can be attained by randomizing $M-1$ ULQs if the pdfs $f_0$,\dots, $f_{M-1}$ are linearly independent. However, this may not be true for $l\ge 3$. Fortunately, since the maximin quantizers can always be approximated as described in Theorem \ref{the:MaximinULQmultialphabets}, the issue is not essential from the viewpoint of numerical computation, as  we can compute  the maximin quantizers (or their approximations) in the same way as  in Subsection \ref{subs:SearchMaximinQntzr} except that each ULQ is now associated with an $l$ by $M$ matrix $A=\{a_{i,m}\}$.

Another benefit of Theorem \ref{the:MaximinULQmultialphabets} is that it can deal with the case when the sensor messages are binary but the pdf's are not linearly independent. Such a case was not addressed by Theorem \ref{the:MaximinQntzrApproxByULQ}, and Theorem \ref{the:MaximinULQmultialphabets} shows that although the maximin quantizer $\bar{\phi}^{\textrm{max}}_m$ may no longer be a randomization of at most $M-1$ ULQs, it can still be  approximated by a sequence of qnatizers $\{\bar{\phi}_{m,i}\},$ each one randomizing at most $M-1$ ULQs.


\subsection{Multiple Sensors}\label{subs:mulSnr}

We now assume that there are $K \ge 2$ sensors in the system in which all raw observations are independent from sensor to sensor conditioned on each $\textbf{P}_m$, $m=0,1,\dots, M-1$. In the following notation, we use the superscripts to denote different sensors
as in Section \ref{sec:pbfm}. For simplicity, we assume the sensor messages are binary, since the extension to the scenario with a finite alphabet sensor  messages can be easily done as in Subsection \ref{subs:non2srmsg}. The key to extend our results  is to treat the quantizers in Sections \ref{sec:2stage} and \ref{sec:main} as vectors of quantizers. Specifically, a (deterministic) quantizer vector is $\phi=(\phi^1,\dots,\phi^K),$ where each local sensor $S^k$ uses the deterministic quantizer  $\phi^k$ to  quantize the raw data.   Denote by  $\Phi^{(K)}$ the set of all (deterministic) quantizer vectors, and define a randomized quantizer vector
\[
\bar{\phi}=\sum_j p^j \phi^{\cdot,j}
\]
where $\phi^{\cdot,j}=(\phi^{1,j},\dots,\phi^{K,j}) \in \Phi^{(K)},$ and $\{p^j\}$ are the probability masses assigned to the set of deterministic quantizer vectors $\{\phi^{\cdot,j}\}\subset\Phi^{(K)}$. Let the set of all quantizer vectors be $\overline{\Phi^{(K)}}$ (a deterministic quantizer can be viewed as a randomized one which assigns probability one to itself).
The implementation of a randomized quantizer vector $\bar{\phi}=\sum p^j\phi^{\cdot,j}$ is the same as that in Subsection \ref{subs:rndmqntzr}, i.e., the fusion center knows about which deterministic quantizer vector is picked, either letting the fusion center conduct the randomization directly or using the pseudo-randomization block design at the local sensor level. Likewise, for a deterministic quantizer vector $\phi=(\phi^1,\dots,\phi^K),$ the K-L divergence of state $m'$ from state $m$ is defined as
\begin{eqnarray} \label{KLvector1}
I(m,m';{\phi})=\sum_{k=1}^{K}I(m,m';{\phi^k})
\end{eqnarray}
and for a randomized quantizer vector $\bar{\phi}=\sum_j p^j \phi^{\cdot,j}$, the K-L divergence is a weighted average as in Section \ref{sec:pbfm}:
\begin{eqnarray} \label{KLvector2}
I(m,m';{\bar{\phi}})=\sum p^j I(m,m';{\phi^{\cdot,j}}).
\end{eqnarray}
Now the maximin quantizer vectors $\{\bar{\phi}^{\textrm{max}}_{m}\}$ and maximin information numbers $\{I(m)\}$ for quantizer vectors can be defined in exactly the same way as in Subsection \ref{subs:rdqntzInfoNmbr}, and the theories developed for single-sensor networks, i.e., Theorems \ref{the:asymp2stageproc}-\ref{the:MaximinQntzrApproxByULQ}, also hold for the multiple sensor cases except replacing the quantizers by quantizer vectors.

A special case is when the sensors are homogeneous, i.e., when the observations are independent and identically distributed among different sensors. In this case,  the maximin quantizer vectors are simply replicates of the maximin quantizers in the corresponding single-sensor case, and such results are summarized in the following proposition.

\begin{proposition}\label{prop:SrsIIDMxminVctrRpl}
Assume that $f_m^1=\dots=f^K_m=f_m$ for $m=0,1,\dots, M-1.$ Fix a state $m$, let $\bar{\phi}^{0,\textrm{max}}_m=\sum_j p^{j}_m\phi^{0,j}_m$ be the maximin quantizer in the corresponding single sensor case where the system has only one sensor and the raw data are distributed according to $\{f_m\}$. Define randomized quantizer vector $\bar{\phi}_m^{*} =\sum_j p^{j}_m \phi^{\cdot,j}_m$ with each $\phi^{\cdot,j}_m$ being a $K$-time replication of $\phi^{0,j}_m$, i.e., $\phi^{\cdot,j}_m=(\phi^{0,j}_m,\dots,\phi^{0,j}_m).$
Then $\bar{\phi}_m^{*}$ is a maximin quantizer vector for the state $m.$
\end{proposition}

\begin{proof}  The proof follows at once from (\ref{KLvector1}) and (\ref{KLvector2}).
\end{proof}

%

\subsection{Composite Multihypothesis Testing}\label{subs:compmulti}

Our theory can also be extended to the scenario of composite hypothesis with finitely many points. Suppose that there
are $B$ composite hypotheses, $\textbf{H}_0$,\dots, $\textbf{H}_{B-1},$ where
\begin{equation*}
  \textbf{H}_{b}=\{\textbf{P}_{i_b}, \textbf{P}_{i_b+1},\dots, \textbf{P}_{i_{b+1}-1}\}
\end{equation*}
include $i_{b+1} - i_{b}$ points for $b=0,1,\dots,B-1,$ and $i_0=0.$ Without loss of generality, let us assume $M = i_{B}.$
Then there are a total of $i_B =M$ simple hypotheses, and the decision maker is required to
pick up one of the $B$ hypotheses that most likely includes the true state of nature $\textbf{P}_m.$
Hence, the problem formulation is the same as that in Section II, except that the cost function $W(m,m')$ needs to be re-defined to reflect composite hypotheses in the multihypothesis testing problem.
To simplify  our notation, for $m=0,1,\dots,M-1$, denote by $[m]$ the hypothesis that $\textbf{P}_m$ is in, i.e., $[m]=\textbf{H}_b$ if and only if $\textbf{P}_m\in \textbf{H}_b$. In composite multihypothesis testing problem, the loss function $W$ has the form $\{W(m,[m'])\}$, where $W(m, [m'])$ indicates the loss caused by making a decision $D=[m']$ when the states of nature is $\textbf{P}_m$. We assume $W(m,[m'])\ge 0$ and $W(m, [m'])=0$ if and only if $m\not\in [m']$, i.e., no loss in making a correct decision.

As in Section \ref{sec:pbfm}, the total expected cost or risk of a test $\delta$ when the true state of nature is $m$ is:
\begin{displaymath}
	\mathcal{R}_c(\delta; m)=c\textbf{E}_m\{N\} + \sum_{[m']}W(m,[m'])\textbf{P}_m\{D=[m']\}
\end{displaymath}
and the Bayes risk of $\delta$ is
\begin{equation}\label{equ:BysRskSeqTestComp}
  \mathcal{R}_c(\delta)=\sum_m\pi_m \mathcal{R}_c(\delta; m)
\end{equation}
where the prior probability of the hypothesis  $\textbf{H}_{b}$ is $\pi_{i_{b}} + \ldots + \pi_{i_{b+1}-1}.$

In the scenario of composite hypotheses, the definition of the two-stage tests is similar except a slight modification of the stopping time $N$ and the final decision $D$  of the fusion center in the second stage.  For simplicity, let us consider the simplest case of the single-sensor and binary sensor messages. At time step $n$ in the second stage, the fusion center computes
\begin{displaymath}
  r_{[m],n}=\sum_{m'\not\in [m]} \pi_{m',n} W(m',[m])
\end{displaymath}
which is the average loss if one makes  a final decision $D=[m].$ Then the fusion center stops at time $N=\min\{N_{[m]}\},$
where
\begin{equation*}
  N_{[m]}=\{n\ge N_0: r_{[m],n}\le c\}
\end{equation*}
and $N_0$ is the stopping time for the first stage.  When stopped, the fusion center makes a final decision $D=[m]$ if $N=N_{[m]}.$
Note that we do not change the fusion center policies in the first stage, i.e., the preliminary decision $D_0$ at the fusion center still picks up the most promising state among the $M$ states instead of picking up one of the $B$ hypotheses.

To find the asymptotically optimal tests among the two-stage tests, we need to modify the definition of the information number $I(m; \bar{\phi})$ as follows:
\begin{equation*}
I(m;{\bar{\phi}})=\min_{m'\not\in[m]}I(m,m';{\bar{\phi}}), \quad m=0,1,\dots, M-1
\end{equation*}
that is, when taking the minimum, we shall ignore those states grouped into the same hypothesis with $m$.
With these new definitions, Theorems \ref{the:asymp2stageproc} and \ref{the:asympoptdeltaIc} remain valid, and we can still use Theorem \ref{the:MaximinQntzrApproxByULQ} to numerically compute each maximin quantizer $\bar{\phi}^{\textrm{max}}_m$ by pretending $[m]=\{\textbf{P}_m\},$ i.e., by temporarily discarding other states in $[m].$

\section{Examples}\label{sec:Exp}

In this section we illustrate our theory via a numerical simulation study. Suppose we are interested in testing the mean of a normal distribution with unit variance in a network with a single sensor and binary sensor messages. That is, the raw data observed at the local sensor follows a normal distribution $\textbf{P}\sim N(\theta, 1).$ In the problem of testing three hypotheses regarding $\theta,$ say, $\textbf{H}_0: \theta = \theta_0,$ $\textbf{H}_1: \theta=\theta_1$ and $\textbf{H}_1: \theta=\theta_2,$ we assign the prior probability of $1/3$ to each of these three hypotheses, and as in Draglin et al. \cite{dtv2}, we also assume 0-1 loss for decision-making, i.e., $W(m,m')=1$ if $m\neq m'$ and $=0$ if $m=m'.$ Two different scenarios will be considered:
\begin{itemize}
  \item[1)] {\bf Asymmetric (HT1):}\ $(\theta_0, \theta_1, \theta_2) = (-0.5, 0, 1).$
  \item[2)] {\bf Symmetric (HT2):}\ $(\theta_0, \theta_1, \theta_2) = (-0.5, 0, 0.5).$
\end{itemize}

For our proposed asymptotic optimal decentralized test $\delta_{A}$ in these scenarios, it suffices to determine the local quantizers. The stationary quantizer in the first stage of $\delta_{A}$ is easy, as we can simply use $\phi^0(X) =I(X\ge 0),$ which satisfies the conditions in Subsection \ref{subs:2stage}. It is a little more challenging to  characterize the maximin quantizers used in the second stage of $\delta_{A}.$ For the asymmetric case (HT1), it is straightforward to show from Lemma \ref{lem:rdccomplexitesInSolveMmQntzr} that the three maximin quantizers are all deterministic MLRQs. Numerical computations illustrate that the three maximin quantizers are $\phi_0=I(X\ge -0.3963), \phi_1 = I(X \ge -0.1037), \phi_2 = I(X \ge 0.7941)$ and the corresponding maximin information numbers are $I_0 = 0.0796,   I_1 = 0.0796, I_2 = 0.3186,$ respectively.

The maximin quantizers of the symmetric case (HT2) are a little tricky. It is easy to check that Lemma \ref{lem:rdccomplexitesInSolveMmQntzr} can be applied to state $m=0$ and $m=2$, yielding two maximin quantizers $\phi_0 = I(X \ge -0.1037 )$ and $\phi_2 = I(X \ge 0.3963)$ with maximin information numbers $I_0=I_2=0.07959$. However, we need to pay special attention to the maximin quantizer for the state $m=1$ since the other two states $m=0$ and $m=2$ are symmetric with respect to $m=1$. Since the three pdfs are obviously linearly independent as defined in Subsection \ref{subs:SearchMaximinQntzr}, by Theorem \ref{the:MaximinQntzrApproxByULQ}, the maximin quantizer for state $m=1$ can be realized as a randomization of at most two ULQs. The following lemma, whose proof is straightforward and thus is omitted, gives more convenient descriptions of the ULQs in (HT2) when the observations are normally distributed.

\begin{lemma} \label{lem:ULQs4HT2}
For the symmetric case (HT2), up to a permutation of the values it takes, a ULQ always takes one of the following two forms:
$I(X \ge \lambda)$ or $I(\lambda_1 \le X\le \lambda_2),$
  where $\lambda$ and $\lambda_1\le \lambda_2$ are real numbers.
\end{lemma}
This allows us to do numerical computation of the maximin quantizer for state $m=1$ as in Subsection \ref{subs:SearchMaximinQntzr}. Numerical computations turns out to show that  the maximin quantizer for state $m=1$ is also the deterministic quantizer defined by $\phi_1=I(X>0)$ up to the precision of 5 decimal digits, and $I_1=0.07928$.

%


For each of two scenarios, (HT1) and (HT2), we will consider two versions of our proposed tests: one is $\delta_{A}(c)$ for the system with a single sensor, and the other is $\delta'_A(c)$ for the system with two independent and identical sensors. As a comparison of our proposed tests, we also consider an asymptotically optimal {\it centralized} test $\delta_a$ proposed in Draglin et al. \cite{dtv, dtv2}
for the system with a single sensor (we omitted another family of asymptotically optimal {\it centralized} test $\delta_b$ proposed in  Draglin et al. \cite{dtv, dtv2}, since its performance is similar to that of $\delta_{a}$). For $\delta_a$, the fusion center updates the posterior distribution $\{\pi_{m,n}\}$ based on the raw data $\{X_n\}$ and its stopping time is defined as $N(a)=\min_{1\le m\le M} N_m(a)$, where $N_m(a)=\inf\{n\ge 1: \pi_{m, n}\ge A_m\}$. In other words, $\delta_a$ stops as soon as one of the posterior probability $\pi_{m,n}$ passes the threshold $A_m$, which can take different values for different $m$. In the numerical simulation given in \cite{dtv2}, the values of these thresholds are as follows. For the asymmetric (\textbf{HT1}), $A_0=A_1=1-3.99\times 10^{-3}$, $A_2=1-5.33\times 10^{-3}$. For the symmetric (\textbf{HT2}), $A_0=A_1=A_2=1-3.99\times 10^{-3}$. These particular values for the thresholds tune the overall probabilities of making incorrect decisions with test $\delta_a$ to $1.0\pm 0.1 \times 10^{-3}$.


In our simulations, the cost $c=3.6\times 10^{-3},$ and the threshold $u(c)$ at the first stage of  our proposed tests $\delta_{A}(c)$ and $\delta'_A(c)$ is set as $0.1.$ Because of the selection of the parameters,  $\delta_A$, $\delta'_A,$ and $\delta_a$ have similar probabilities of making incorrect decisions, i.e., $1.0\pm 0.1 \times 10^{-3}.$ Thus it suffices to report  the simulated expected time steps $\textbf{E}_m\left\{ N \right\}$ under each of the three hypotheses $\textbf{H}_{m}$ for $m=0,1,2,$ as smaller values of  $\textbf{E}_m\{N\}$ imply better performance of the test (in the sense of smaller Bayes risks). These results are reported in Table \ref{tab:sampleSizecmp}.

\begin{table}
  \centering
  \caption{Expected values of time steps taken for each of the three tests.} \label{tab:sampleSizecmp}

\begin{tabular}{|c||c||c|c|c|}
\hline
& $\textbf{E}_m(N)$ & $\delta_a$ & $\delta_A(c)$ & $\delta'_A(c)$  \\
\hline
\hline
& $m=0$ & 46.48  & 73.5$\pm$0.9& 36.8$\pm$0.7\\
	 \cline{2-5}
Asymmetric (HT1) &$m=1$ & 48.39 & 77.7$\pm$0.9& 38.9$\pm$0.7\\
	 \cline{2-5}
&$m=2$ & 11.90 & 19.8$\pm$0.2& 9.9$\pm$0.1\\
	 \hline \hline
&$m=0$ & 46.59  & 73.4$\pm$0.9& 37.8$\pm$0.6\\
	 \cline{2-5}
Symmetric (HT2) &$m=1$ & 69.43 & 110.2$\pm$0.9& 55.2$\pm$0.7\\
	 \cline{2-5}
&$m=2$ & 46.60 & 73.4$\pm$0.9& 37.8$\pm$0.6\\  \hline
\end{tabular}


\end{table}

The numerical results illustrate that the centralized test, $\delta_{a},$ indeed performs better than the decentralized test $\delta_{A}(c)$ that makes a final decision based on binary sensor messages instead of raw normal observations. However, Table \ref{tab:sampleSizecmp} demonstrates that for (HT1) and (HT2), if we are able to deploy merely one extra identical sensor, the decentralized test $\delta'_{A}(c)$  has smaller Bayes risk than the centralized test with a single sensor, not to mention other important benefits such as robustness and bandwidth saving capabilities. In other words, if designed appropriately, a decentralized test does not lose much information as compared to the centralized test, and in fact, a decentralized test with two sensors can outperform a centralized test with a single sensor.

\section{Conclusion} \label{sec:conclusion}

We have developed a family of  asymptotically optimal decentralized sequential tests when testing $M \ge 3$ hypotheses.
The main idea is to consider ``two-stage'' tests in which one first uses a small portion of total time steps to make a preliminary decision of the true state of nature, and then the local quantizer switches to the corresponding ``maximin quantizers.'' Moreover, we show that each maximin quantizer can be realized (or approximated) as a randomization of at most $M-1$ ULQs, and we also illustrate how to compute maximin quantizers numerically.

There are several theoretical issues in sequential multihypothesis testing problems that deserve further research. Instead of first-order optimality, it will be interesting to investigate higher-order asymptotic optimality. It is expected that we need to extend our two-stage test $\delta_{A}(c)$ to more than two stages in order to achieve higher-order asymptotic optimality. In addition, it is interesting to see what happens if the sensor observations are no longer i.i.d., especially if they are dependent either over time or among different sensors.

\appendices

\renewcommand{\thelemma}{\thesection.\arabic{lemma}}

\section{Proofs of Theorems \ref{the:MaximinQntzrApproxByULQ} and \ref{the:MaximinULQmultialphabets}}\label{app:qntzr}

Since quantizers, especially randomized quantizers, play an important role in our theorems,  we will gather some useful results for quantizers in this appendix, including the proofs of Theorems \ref{the:MaximinQntzrApproxByULQ} and \ref{the:MaximinULQmultialphabets}.  Without loss of generality, we assume that the quantized messages belong to a finite alphabet, say, $\{0,1,\ldots, l-1\}.$
For a (deterministic or randomized) quantizer $\bar\phi \in \bar\Phi$, define  its distribution vector as a vector of $Ml$ dimensions:
\[
q(\bar\phi)=(q(i; m,\bar\phi))_{\ 0\le i\le l-1;\ 0\le m\le M-1}
\]
where $q(i; m, \bar\phi)=\textbf{P}_m(\bar\phi(X)=i).$ Now let us consider four subspaces induced by the distribution vectors $q(\bar\phi):$
\begin{itemize}
  \item Let $Q$ be the set formed by the distribution vectors of all \textit{deterministic} quantizers, i.e., $Q=\{q(\phi):\phi\in\Phi\};$
  \item Let $\bar Q=\{q(\bar\phi): \bar\phi\in\bar\Phi\}$ be the set formed by the distribution vectors of all quantizers, deterministic or random;
  \item Denote by $Q_U\subset Q$ the set of distribution vectors of all ULQs (see Definition \ref{def:ULQGeneral});
  \item Denote by $Q_\alpha$ the set of extreme points of $\bar{Q}$.
\end{itemize}
By Tsitsiklis \cite{tsi}, $Q$ is compact and $\bar Q$ is the compact convex hull of $Q.$  By the Krein-Milman theorem, the compact convex set $\bar Q$ is also the convex hull of its extreme points. Thus it is useful to characterize $Q_\alpha.$ Tsitsiklis \cite{tsi} showed that $Q_U \subset Q_\alpha \subset Q,$ and $Q_{U}$ is a dense subset of $Q_\alpha.$ Moreover, it also studied in detail the case of
testing $M=2$ hypotheses. However, the case of $M \ge 3$ hypotheses is more challenging.  Fortunately, below we are able to show that $Q_\alpha = Q_{U}$ for $M \ge 3$ hypotheses under some reasonable additional assumptions.

\begin{lemma} \label{lem:ExpisExt}
If the sensor messages are binary (i.e., $l=2$) and the pdf's $\{f_0,\dots,f_{M-1}\}$ are linearly independent (as defined in Subsection \ref{subs:SearchMaximinQntzr}), then $Q_\alpha=Q_U.$
\end{lemma}

\begin{proof} Since  $Q_U$ is a dense subset of $Q_\alpha,$ it is sufficient to show that if $q^0 \in Q_{\alpha},$ then $q^0 \in Q_{U}.$ Since $Q_{U}$ is dense in $Q_\alpha,$ there is a sequence of ULQs $\phi^j,$ say, $\phi^j=I(\sum_m a^j_m f_m(X)>0)$ with $\sum_m (a^j_m)^2 =1,$ such that $q(\phi^j) \to q^0.$ By Bolzano-Weierstrass theorem, each bounded sequence has a convergent subsequence. By passing to subsequences, we can simply assume that $a^{j}_{m}$ converges to $a^{0}_{m}$ for each state $m,$ and so $\sum_m (a^0_m)^2 =1$. By the condition of linear independence, $\phi^0(X)=I(\sum_m a^0_m f_m(X)>0)$ is a ULQ. It remains to show that $q^0=q(\phi^0)$, or equivalently,  to show that for each state $m$,
$\lim_{j\to\infty}\textbf{P}_m(\Delta_j) = 0,$
where $\Delta_j=\Delta_j^1\cup \Delta_j^2,$ and
  \begin{displaymath}
	\Delta_j^1=\{X: \sum_m a_m^j f_m(X)\le 0 \quad \mbox{and} \quad \sum_m a_m^0 f_m(X)>0\}
  \end{displaymath}
  and
  \begin{displaymath}
	\Delta_j^2=\{X: \sum_m a_m^j f_m(X)> 0 \quad \mbox{and} \quad  \sum_m a_m^0 f_m(X)\le 0\}.
  \end{displaymath}

To prove this, without loss of generality,  let us further assume that $0 \le f_m(X)\le 1$ for any state $m,$
as we can always substitute $f_m(X)$ by ${f_m(X)}/{\sum_{m'} f_{m'}(X)}.$ Define
another sequence of sets $\{\Delta'_j\}$ by $\Delta'_j=\{X: |\sum_m a_m^0 f_m(X)|\le M\varepsilon_j\},$
where $\epsilon_{j} = \max_{m} |a^{0}_{m} - a^{j}_{m}|.$ We claim that $\Delta_j\subset \Delta'_j$ for each $j$. Indeed, if $X\in\Delta^1_j$, then $\sum_m a_m^j f_m(X) \le 0$ and
 \begin{eqnarray*}
	\sum_m a_m^0 f_m(X) &\le& \sum_m (a_m^0-a_m^j)f_m(X)\\
&\le& \sum_m |a_m^0-a_m^j| = M\varepsilon_j
  \end{eqnarray*}
where the second inequality uses the assumption that $0 \le f_m(X)\le 1.$  Moreover, if $X\in\Delta^1_j$, then $\sum_m a_m^0 f_m(X) > 0,$ and thus $\Delta^1_j\subset \Delta'_j.$ Similarly, $\Delta^2_j\subset\Delta'_j.$ So $\Delta_j\subset\Delta'_j$.

Let $\Delta^0=\bigcap_{i=1}^\infty \bigcup_{j=i}^\infty \Delta'_j$. Since $a^{j}_{m}$ converges to $a^{0}_{m}$ for each state $m,$
we have $\varepsilon_j= \max_{m} |a^{0}_{m} - a^{j}_{m}| \to 0,$ and thus $\Delta^0=\{X: \sum_m a_m^0 f_m(X)=0\}$. Because the pdf's are assumed to be linearly independent, $\textbf{P}_m(\Delta^0)=0$ for any state $m$. Hence, $\lim_{j \rightarrow \infty} \textbf{P}_m(\Delta'_j)=0.$ So $\lim_{j \rightarrow \infty} \textbf{P}_m(\Delta_j)=0,$ and the lemma is proved.
\end{proof}

\medskip
Now let us consider the K-L divergences for distribution vectors of quantizers. 
Given $q\in \bar Q$, say, $q=q(\bar\phi)$, denote $q_{i,m}=q(i; m, \bar\phi)$, where $i=0,\dots,l-1$ and $m=0,\dots,M-1$. For $0\le m\neq m'\le M-1,$ define the K-L divergence of the distribution vector $q$ of state $m'$ from state $m$ by
\begin{equation}\label{equ:Immpmq}
	J(m,m'; q)=\sum_{i=0}^{l-1} q_{i,m}\log\frac{q_{i,m}}{q_{i,m'}}
\end{equation}
where as conventional $0\log\frac{0}{0}=0.$

On the one hand, the definition of $J(m,m'; q)$ is standard and Tsitsiklis \cite{tsi} showed that under Assumption \ref{ass:KLInfoLimit}, for any two states $m\neq m'$, the K-L divergence   $J(m,m';{q})$  is bounded, continuous, and convex as a function of $q\in \bar{Q}.$ On the other hand, for a randomized quantizer $\bar\phi,$ the definition of  $J(m,m'; q(\bar\phi))$ is equivalent to  the K-L divergence defined in (\ref{equ:inforndma}), not that in (\ref{equ:inforndm}). Indeed, $J(m,m'; q(\bar\phi)) \le I(m,m'; \bar \phi)$ in (\ref{equ:inforndm}) and thus it does not directly relate to the maxmin information number $I(m)$ in Definition \ref{def:maximin-info}. Fortunately, the idea can be salvaged. To do so, let $\bar{\mathcal{M}}$ be the set of Borel probability measures on $\bar{Q}$, for each $\mu\in\bar{\mathcal{M}}$ and two states $0\le m\neq m'\le M-1$ define
\begin{equation}\label{equ:ImmprimMu}
J^{*}(m,m'; \mu)=\int_{\bar Q} J(m,m'; q)d\mu(q)
\end{equation}
and
\begin{equation}\label{equ:ImMu}
J^{*}(m; \mu)=\min_{m'\neq m} J^{*}(m, m'; \mu).
\end{equation}
Then for a randomized quantizer $\bar{\phi}\in\bar{\Phi},$ the K-L divergence defined in (\ref{equ:inforndm}) is equivalent to $J^{*}(m,m'; \mu)$ for some suitably chosen $\mu.$ To see this, note that $\bar{\phi}$ assigns probability masses to a finite or  countable subset of $\Phi,$ and thus induces a probability measure $\mu(\bar{\phi})$ on $Q.$  Hence, $I(m,m'; \bar{\phi})=J^*(m,m'; \mu(\bar{\phi}))$ and
\begin{equation}\label{equ:ImbarphiJ*}
I(m; \bar{\phi})=J^*(m; \mu(\bar{\phi})).
\end{equation}




Our next result is to provide an alternative representation of the maximin information number $I(m)$ defined in Definition \ref{def:maximin-info} in Subsection \ref{subs:rdqntzInfoNmbr}.
\begin{lemma} \label{lem:ImmuSupIsIm}
The maximin information number $I(m) = \sup_{\mu\in\mathcal{M}}J^*(m; \mu)=\sup_{\mu\in\bar{\mathcal{M}}}J^*(m; \mu),$ where  $\mathcal{M}\subset\bar{\mathcal{M}}$ is the set of probability measures supported on $Q.$
\end{lemma}
\begin{proof} Denote  by $\mathcal{M}^0$ and  $\bar{\mathcal{M}}^0$ the set of probability measures on $Q$ and $\bar Q$ that have at most countable supports, respectively. By (\ref{equ:ImbarphiJ*}), $\sup_{\mu\in\mathcal{M}^0}J^*(m; \mu)= I(m),$ and thus
\[ I(m)\le \sup_{\mu\in\mathcal{M}}J^*(m; \mu)\le\sup_{\mu\in\bar{\mathcal{M}}}J^*(m; \mu).\]
By Tsitsiklis \cite{tsi}, $J(m,m'; q)$ is bounded and continuous  as a function of $q\in \bar Q$. Hence $J^*(m, m'; \mu)$ and $J^*(m; \mu)$ are also continuous viewed as functions of $\mu\in\bar{\mathcal{M}}$ (under weak-convergence). Thus the lemma follows at once from the denseness of
$\mathcal{M}^0$ (or $\bar{\mathcal{M}}^0$)  in $\mathcal{M}$ (or $\bar{\mathcal{M}}$), provided that $I(m)\ge\sup_{\mu\in\bar{\mathcal{M}}^0}J^*(m; \mu).$ Hence, it suffices to show that for each $\mu\in\bar{\mathcal{M}}^0$, there exists a $\mu'\in\mathcal{M}^0$ such that $J^*(m, m'; \mu)\le J^*(m, m'; \mu')$ for each $m'\neq m$. By linearity, we only need to prove it under the further assumption that $\mu\in\bar{\mathcal{M}}^0$ is supported on a single point $q =q(\bar{\phi})$ for a randomized quantizer $\bar{\phi}\in\bar\Phi.$ In this case $J^*(m, m'; \mu) = J(m, m'; q) \le I(m, m'; \bar\phi)$.  By our previous argument, $\bar{\phi}$ can be identified to a probability measure $\mu'=\mu(\bar \phi)\in \mathcal{M}^0$ with the property $I(m, m'; \bar\phi)=J^*(m,m'; \mu')$. Therefore $J^*(m,m'; \mu) \le J^{*}(m,m'; \mu'),$ completing the proof of the lemma.
\end{proof}

\medskip

\medskip
Finally, we are in a position to prove Theorems \ref{the:MaximinQntzrApproxByULQ} and \ref{the:MaximinULQmultialphabets}.

\begin{proof}[Proofs of Theorem \ref{the:MaximinQntzrApproxByULQ} and Theorem \ref{the:MaximinULQmultialphabets}]
Note that Theorem \ref{the:MaximinQntzrApproxByULQ} is a special case of Theorem \ref{the:MaximinULQmultialphabets}, and follows at once from Theorem \ref{the:MaximinULQmultialphabets} and Lemma \ref{lem:ExpisExt} under the assumption of binary sensor messages and linearly independent pdf's in which $Q_U=Q_\alpha.$ By symmetry and the fact that $Q_U$ is a dense subset in $Q_\alpha$, it is sufficient to show that under the assumption of Theorem \ref{the:MaximinULQmultialphabets}, for the state $m=0$, exists one maximin quantizer which is a randomization of at most $M-1$ quantizers with their distribution vectors in $Q_\alpha$.

Define two sets in $M-1$ dimensional space,
$\mathscr{I} = \{(J(0,1; q), \dots, J(0, M-1; q) ): q\in Q\},$ and $\mathscr{I}_{\alpha}= \{(J(0,1; q), \dots, J(0, M-1; q) ): q\in Q_\alpha\}.$ Define the same for $\mathscr{I}^{*}$ and $\mathscr{I}_{\alpha}^{*}$ when $J(0,m;q)$ is replaced by $J^{*}(0,m; \mu)$ with $\mu\in\mathcal{M}$ and $\mu \in \mathcal{M}_\alpha,$ respectively, where $\mathcal{M}_\alpha$ is the set of probability measures supported in $Q_\alpha.$ As we have mentioned earlier, $J(0, m; q)$ is continuous if viewed as a function of $q\in Q$, so both $\mathscr{I}$ and $\mathscr{I}_\alpha$ are compact. Obviously, $\mathscr{I}^*$ and $\mathscr{I}^*_\alpha$ are convex hulls of $\mathscr{I}$ and $\mathscr{I}_\alpha$, so they are compact as well. The main idea of the proof is to relate the maximin information number $I(0)$ with the set  $\mathscr{I}^{*}_{\alpha}.$

First, we claim that $I(0) = \sup_{J\in \mathscr{I}^*_\alpha} h(J),$ where $h(\cdot)$ is a function on $M-1$ dimensional space defined by $h(x_1,\dots, x_{M-1})=\min\{x_1,\dots,x_{M-1}\}.$  By Lemma \ref{lem:ImmuSupIsIm}, we have $I(0)=\sup_{J\in \mathscr{I}^*} h(J).$
Since $\mathscr{I}^*_\alpha\subset\mathscr{I}^*$, to prove the claim, we only need to show, for any $J\in\mathscr{I}^*$, there exists $J'\in\mathscr{I}^*_\alpha$, such that each component of $J$ is less or equal to the corresponding component of $J'$. By linearity, it is sufficient to prove for $J\in\mathscr{I},$ say, $J=(J(0,1; q),\dots,J(0,M-1; q))$ for some $q\in Q$. Decompose $q$ as a convex combination of points in $Q_\alpha$: $q=\sum p^j q^j$, then
  \begin{displaymath}
	J(m, m'; q)\le \sum p^j J(m, m'; q^j), \quad 0\le m\neq m'\le M-1.
  \end{displaymath}
  Let $J'=(J^*(0,1; \mu),\dots,J^*(0,M-1; \mu))$ with $\mu$ assigns probability mass $p^j$ to $q^j$ for each $j$, and our claim is justified.

Second, we will show that
\begin{equation}\label{equ:JhJ}
\sup_{J\in\mathscr{I}^*_\alpha} h(J) = \min_{1\le m\le M-1}J^*(0,m; \mu_0)
\end{equation}
for a probability $\mu_0\in \mathcal{M}_\alpha$ whose support includes at most $M-1$ points.
To see this,
note that  $\mathscr{I}^*_\alpha$ is a compact convex subset in $M-1$ dimensional space. Thus $h(\cdot)$ attains its maximum at a point $\tilde{J}$ on the surface of $\mathscr{I}^*_\alpha$ and $\tilde{J}$ can be realized as a convex combination of at most $M-1$ points in $\mathscr{I}^*_\alpha,$ see, for example, Hormander \cite{hor}. Suppose that $\tilde{J}=\sum_{j=1}^{M-1} p_0^j J^j$, where $\sum p_0^j=1$ and $J^j\in\mathscr{I}^*_\alpha$. For each $j$, let $J^j=(J(0,1; q^j_0),\dots,J(0,M-1; q^j_0))$, with $q^j_0\in Q_\alpha$. Define $\mu_0\in\mathcal{M}_\alpha$ be a probability measure such that $\mu_0(q^j_0)=p^j_0$, for $j=1,\dots, M-1$, then (\ref{equ:JhJ}) holds.

Finally, define the randomized quantizer $\bar{\phi}_0$ as the one induced by the measure $\mu_0$ in (\ref{equ:JhJ}). Then $I(0)=\min_{m\neq 0}\{I(0,m; \bar{\phi}_0)\}$ and $\bar{\phi}_0$ can be rewritten as $\sum_{j=1}^{M-1}p_0^j \phi_0^j$ where $\phi_0^j$ has $q_0^j$ as its distribution vector. Equivalently, $\bar{\phi}_0$ is just the maximin quantizer $\bar{\phi}^{\textrm{max}}_0,$ and it can be taken as a randomization of at most $M-1$ quantizers with their distribution vectors in $Q_\alpha$. This completes our proof.
\end{proof}


\section{Proof of Theorem  \ref{the:asymp2stageproc} }  \label{app:theproof}

At each stage of our proposed two-stage test $\delta(c),$ since the local sensor uses stationary (though possibly randomized) quantizers, the sensor messages $U_{n}$'s are i.i.d. and the fusion center essentially faces the classical centralized sequential hypothesis testing problems. Thus Theorem \ref{the:asymp2stageproc} can be proved by standard arguments and by conditioning on the preliminary decision $D_0$ of the two-stage test $\delta(c).$ In the following we will focus on the proof of (\ref{equ:samplesize2stageproc}) to highlight the associated technical mathematical problems that need special attention. Denote by $N_0$ and $N_1$ the total time steps of the first and second stages of the two-stage test $\delta(c),$ respectively,
then the total time step $N$ taken by $\delta(c)$ satisfies
\begin{eqnarray*}
\textbf{E}_m\left\{ N \right\} &=&  \textbf{E}_m\left\{ N_0 \right\}+\textbf{E}_m\left\{ N_1 \right\} \nonumber \\
		&=& \textbf{E}_m\left\{ N_0 \right\}+\textbf{E}_{m}\left\{ N_1|D_0=m \right\}\textbf{P}_m\left\{ D_0=m \right\}\\
		&&+\textbf{E}_m\left\{ N_1 1\{D_0\neq m\} \right\}. \label{equ:EmN-decmp}
\end{eqnarray*}
By standard arguments for the classical centralized sequential multihypothesis testing problems, the stopping boundary of $1-u(c)$ at the first stage guarantees that $\textbf{P}_m\{D_0=m\} = 1-O(u(c))$ and $\textbf{E}_m\{N_0\} = O(|\log u(c)|).$ Since $u(c) \rightarrow 0$ satisfies $|\log u(c)| / |\log c| \rightarrow 0,$ e.g., $u(c) = 1 /|\log c|,$ we have $\textbf{P}_m\{D_0=m\} = 1-o(1)$ and
$\textbf{E}_m\{N_0\} = o(|\log c|).$ Hence, equation (\ref{equ:samplesize2stageproc}) holds if we can further show that

\begin{equation}\label{eqnJuly8n03} 
\textbf{E}_m\{N_1|D_0=m\} = (1+o(1))|\log c|/I(m; \bar{\phi}_m)    
\end{equation}

and

\begin{equation} \label{eqnJuly8n04}
\textbf{E}_m\{N_1 1\{D_0\neq m\}\} = o(|\log c|). 
\end{equation}
To prove (\ref{eqnJuly8n03}) and (\ref{eqnJuly8n04}), note that at time $n$ of the second stage of our proposed two-stage test $\delta(c),$ the log-likelihood ratio statistic of the latest sensor message at the fusion center is
 \[
 \Delta Z_{n}(m,m'; \phi^{j(n)})=\log \frac{f_{m}(U_n; \phi^{j(n)})}{f_{m'}(U_n; \phi^{j(n)})},
\]
where $\phi^{j(n)}$ is the deterministic quantizer selected through the randomization at time step $n$ and $U_n=\phi^{j(n)}(X_n)$ is the quantized sensor message. Hence, for our proposed two-stage test,  the log-likelihood ratio statistic of all available sensor messages up to time $n$ is
\begin{equation}\label{equ:logsum}
	Z_n(m,m'; \bar{\phi})=\sum_{i=1}^{n}\Delta Z_i(m,m'; \phi^{j(i)}).
\end{equation}
Furthermore, since $\bar{\phi}$ is assumed to be a randomization of a finite number of deterministic quantizers, our implementation of randomized quantizers implies that $\{\Delta Z_{n}(m,m'; \phi^{j(i)}), n=1, 2,\dots\}$ is a sequence of i.i.d. random variables with mean $I(m,m';{\bar{\phi}})$ in (\ref{equ:inforndm}) and finite variance.

To prove (\ref{eqnJuly8n03}), it is sufficient to show that
\[
		\textbf{E}_m\left\{ N_1|D_0=m, \pi_{\cdot, N_0} \right\} = (1+o(1))|\log c|/I(m; \bar{\phi}_m)
\]
where  $\pi_{\cdot, N_0}=(\pi_{0, N_0},\dots,\pi_{M-1, N_0})$ denotes the posterior distributions at time $N_0$ and the $o(1)$ term is uniform on the event $\{D_0=m\}$ for any possible $\pi_{\cdot, N_0}.$ This relation itself follows at once from the fact that $Z_{n}(m,m';\bar{\phi})\}$ is the sum of i.i.d. random variables with mean $I(m,m';{\bar{\phi}})$ in (\ref{equ:inforndm}) and finite variance, but we need some extra work to prove the uniformness of the $o(1)$ term. For that purpose, given the state $m,$ let $B_{c} =|\log c/(1-c)|+|\log (1-u(c))|$ and consider the following stopping time:
\begin{equation}\label{equ:July16eq01}
T(B_{c}; \bar{\phi}_m)=\inf\{n: \min_{m':m'\neq m}Z_n(m,m'; \bar{\phi}_{m} )\ge B_{c} \}
\end{equation}
where  $Z_n(m,m'; \bar{\phi}_{m})$ is the log-likelihood ratio in (\ref{equ:logsum}) except that the quantizer $\bar{\phi}$ is now replaced by $\bar{\phi}_m$ since we condition on $D_0 = m.$ Clearly, under the conditional distribution $\textbf{P}_m\{\cdot |D_0=m, \pi_{\cdot,N_0}\},$ the stopping time $N_1$ is dominated by $T(B_{c}; \bar{\phi}_m)$, which  does not depend on $\pi_{\cdot,N_0}.$ By  the law of large numbers, we have $\textbf{E}_m\{T(B_{c}; \bar{\phi}_m)\}/B_{c} \to 1/I(m; \bar{\phi}_m),$ also see Theorem 5.1 of Baum and Veeravalli \cite{bv}. Thus we can have a $o(1)$ term with the $\le$ part of relation  (\ref{eqnJuly8n03}) due to the above arguments and the fact that $\log u(c)=o(|\log c|).$  The $\ge$ part of the relation can be proved similarly and thus relation  (\ref{eqnJuly8n03}) holds.



The proof of (\ref{eqnJuly8n04}) involves more technical details. It suffices to show that $\textbf{E}_m\{N_1 1\{D_0=m'\}\}=o(|\log c|)$ for each $m'\neq m$. Now when $\{D_0 = m'\},$ our proposed two-stage procedure $\delta(c)$ uses the stationary (likely randomized) quantizer $\bar{\phi}_{m'}$ at the second stage. Hence, we can define $Z_n(m,m'; \bar{\phi}_{m'})$ as in (\ref{equ:logsum}) except that we now use the stationary quantizer $\bar{\phi}_{m'}.$ Likewise, define  $T(B_{c}^{*}; \bar{\phi}_{m'})$ as in (\ref{equ:July16eq01}) with $B_{c}^{*} =|\log c/(1-c)|+|\log \tilde{\pi}_m|,$ where $\tilde{\pi}_m = \pi_{m,N_0}$ is the posterior probability of the $m$th hypothesis at time $N^0.$  Then
\begin{eqnarray*}
		&& \textbf{E}_m\left\{ N_1 1\{D_0=m'\} \right\}\\
		&\le& \textbf{E}_m\left\{T(B_{c}^{*}; \bar{\phi}_{m'}) 1\{D_0=m'\} \right\}  \\
        &=& \textbf{E}_m\left\{ (1+o(1)) B_{c}^{*} 1\{D_0=m'\} /{I(m; \bar{\phi}_{m'})}  \right\} \\
		&\le& (1+o(1))/{I(m; \bar{\phi}_{m'})} \times \\ 
		&&\textbf{E}_m\left\{ (|\log c/(1-c)|+|\log \tilde{\pi}_m|) 1\{D_0=m'\} \right\} \\
		&=& O(|\log c|)\textbf{P}_m\left\{ D_0=m' \right\}+\\
		&&O(1)\textbf{E}_m\left\{ |\log \tilde{\pi}_m| 1\{D_0=m'\} \right\}\\
		&=& o(|\log c|)+ O(1)\textbf{E}_m\left\{ |\log \tilde{\pi}_m  | 1\{D_0=m'\} \right\}.
	\end{eqnarray*}
Thus, to prove (\ref{eqnJuly8n04}), it remains to show that $\textbf{E}_m\left\{ |\log \tilde{\pi}_m | 1\{D_0=m'\} \right\}=o(|\log c|)$ with $\tilde{\pi}_m = \pi_{m,N_0}.$ Below we will prove a stronger statement that \[\textbf{E}_m\left\{ |\log \pi_{m,N_0}| 1\{D_0 \ne m\} \right\}=o(1).\]

By assumption, at time $N_0,$ if $D_0=m'$ then $\pi_{m', N_0} \ge 1 -u(c) > 1/2.$ So $\pi_{m, N_0} < u(c) < 1/2$ and for all $L > 0,$
	\begin{eqnarray*}
		&& \textbf{P}_m\left\{ |\log \pi_{m,N_0}| >L, D_0\neq m\right\}\\
		&\le& \textbf{P}_m\left\{ \log \frac{1-\pi_{m,N_0}}{\pi_{m,N_0}}>L-\log 2, D_0\neq m\right\}\\
		&\le& \textbf{P}_m\left\{\sup_{n\ge 1} \log\frac{1-\pi_{m,n}}{\pi_{m,n}}>L-\log 2\right\}\\
		&\le& \textbf{P}_m\left\{ \sup_{n\ge 1}\sum_{m':m'\neq m}\frac{\pi_{m'}}{\pi_m}\exp\{-Z_n(m,m'; \phi^0)\}\right.\\
		&& \left.\phantom{\sum_{m':m'\neq m}\frac{\pi_{m'}}{\pi_m}}>e^L/2\right\}\\
		&\le& \textbf{P}_m\left\{ \min_{m':m'\neq m} \inf_{n\ge 1} Z_{n}(m,m'; \phi^0)\phantom{\frac{2(M-1)}{\pi_m}}\right.\\
		&&\left.<-L+\log\frac{2(M-1)}{\pi_m} \right\}.
	\end{eqnarray*}
Assume for a moment that the minimum $Z^{*} = \min_{m':m'\neq m} \inf_{n\ge 0} Z_n(m,m';\phi^0)$
is exponentially bounded in the sense that there exists a constant $C_1>0$ and $0<\rho<1$  such that for any $L>0$,
\begin{eqnarray} \label{equ:July16eq02}
\textbf{P}_m\left\{Z^{*} \le -L\right\}\le C_1 \rho^{L}.
\end{eqnarray}
Then we have
\[
\textbf{P}_m\left\{ |\log \pi_{m,N_0}| >L, D_0\neq m\right\} \le C_2 \rho^{L}
\]
with the constant $C_2 = C_1 \exp(-\log \rho \log\frac{2(M-1)}{\pi_m}).$ Consequently,
	\begin{eqnarray*}
		&&\textbf{E}_m\left\{ |\log \pi_{m,N_0}| 1\{D_0\neq m\} \right\}\\
		&=& \textbf{E}_m\left\{ |\log \pi_{m,N_0}| 1\{D_0\neq m, |\log\pi_{m,N_0}|\ge |\log u(c)|\} \right\} \\
		&\le& C_2\int_{|\log u(c)|}^{\infty}\rho^L dL\\
		&=& \frac{C_2}{|\log \rho|}\rho^{|\log u(c)|}
	\end{eqnarray*}
which goes to $0$ as $c\to 0.$ Thus (\ref{eqnJuly8n04}) is proved and the theorem holds.

It remains to prove (\ref{equ:July16eq02}). Since the log-likelihood ratio statistic $Z_{n}(m,m'; \bar \phi)$ in (\ref{equ:logsum}) is the sum of i.i.d. random variables with positive mean and finite variance under ${\bf P}_m,$ the minimum
\[
Z^{*}_{m'} = \inf_{n\ge 0} Z_n(m,m'; \bar \phi)
\]
is a well-defined (non-positive valued) random variable under ${\bf P}_m.$ Moreover,
	\begin{eqnarray*}
		\textbf{P}_m\left\{Z^{*} \le - L\right\} &\le& \sum_{m':m'\neq m}\textbf{P}_m\left\{Z^{*}_{m'} \le -L\right\}.
	\end{eqnarray*}
Thus, to prove (\ref{equ:July16eq02}), it suffices to show that $Z^{*}_{m'}$ is exponentially bounded for each $m'.$ Define a stopping time $\tau_{-} = \inf\{n: Z_{n} (m,m'; \bar \phi) < 0\}$ and let $Y_1, Y_2, \ldots$ be i.i.d. random variables, where $Y_1 = Z_{\tau_{-}}(m,m'; \bar \phi)$ conditional on the event $\tau_{-} < \infty.$ Then it is well-known that $Z^{*}_{m'}$ has the same distribution as $\sum_{i=1}^{\tilde{N}} Y_{i},$ where $\tilde{N}$ is a geometric random variable independent of $Y_{i}$'s such that $P(\tilde{N}=n) = p(1-p)^{n}$ with $p= \textbf{P}_m\{Z^{*}_{m'} =0\}>0,$ see Klass \cite{klass}, or Lemma 11.3 and Remark 11.3 of Gut \cite{gut}. Now in our case, since $U_{n}$ is discrete and $\bar \phi$ is randomization of a finite number of deterministic quantizer, $\Delta Z_{n}(m,m'; \bar \phi)$ has a lower bound, say $-C$ for some $C > 0.$ Thus $Y_1 = Z_{\tau_{-}}(m,m'; \bar \phi)$ also has a lower bound $-C.$ So
\begin{eqnarray*}
\textbf{P}_m\left\{Z^{*}_{m'} \le -L\right\} &=& P(\sum_{i=1}^{ \tilde{N} } Y_{i} \le - L) \\
 &=& P( \tilde{N} \ge L / C ) \\
 &=& (1-p)^{[L/C]}
\end{eqnarray*}
where the last relation uses the fact that $\tilde{N}$ is geometrically distributed. Hence $Z^{*}_{m'}$ is exponentially bounded and the theorem holds. It is also instructive to compare $Z^{*}_{m'}$ with Brownian motion. Let $B(t)$ denote standard Brownian motion with mean zero and variance parameter $1.$ Then for all positive $L, \mu$ and $\sigma,$
\[
{\bf P}( \inf_{t \ge 0}\{\sigma B(t) + \mu t \} \le - L) = \exp(-2 \mu \sigma^{-2} L).
\]

\section{Proof of Theorem \ref{the:asympoptdeltaIc}}\label{app:lwbds}


To prove Theorem \ref{the:asympoptdeltaIc}, the main idea  is to construct a martingale based on log-likelihood ratios and then apply the optional stopping theorem and Wald's inequalities. Since Theorem \ref{the:asympoptdeltaIc} deals with general decentralized sequential tests that may or may not implement randomized quantizers as we proposed for the two-stage tests,
denote by $\tilde{\phi}_n$ the quantizer used at time step $n$ to the best knowledge of the fusion center. For example, when a randomized quantizer $\bar{\phi}=\sum p^j\phi^j$ is implemented and the fusion center knows that the deterministic quantizer $\phi^j$ is picked at time step $n$, then $\tilde{\phi}_n=\phi^j.$ Meanwhile, if the randomization is done at the local sensor and the fusion center has no access about which deterministic quantizer is picked, then $\tilde{\phi}_n=\bar{\phi}.$

Let $U_n$ be the sensor message at time step $n$ and let $q(\tilde{\phi}_n)$ be the distribution vector of $\tilde{\phi}_{n}.$ For $n=1,2,\dots$, define $\mathcal{F}_{n-1}$ as the $\sigma$-algebra generated by $U_1,\dots, U_{n-1}$ and $q(\tilde{\phi}_1), \dots, q(\tilde{\phi}_{n})$.  In other words, $\mathcal{F}_{n-1}$ is all the past information available to the fusion center before the $n$th time step. Then at time step $n,$ the  log-likelihood ratio of state $m$ with respect to state $m'$ is
$Z_n=\sum_{i=1}^{n} \Delta Z_i$, where
\[
\Delta Z_i=\log \frac{f_m(U_i|\mathcal{F}_{i-1})}{f_{m'}(U_i|\mathcal{F}_{i-1})}
\]
and $f_m(\cdot|\mathcal{F}_{i-1})$ is the conditional probability mass function induced on $U_{i}$ under $\textbf{P}_m.$ Since $U_i$ depends on $\mathcal{F}_{i-1}$ only through $\tilde{\phi}_i,$ $f_m(\cdot|\mathcal{F}_{i-1})$ is simply $f_{m}(\cdot; \tilde{\phi})$ in (\ref{equ:inforndma}), and thus $\textbf{E}_m\{\Delta Z_i|\mathcal{F}_{i-1}\}=J(m,m'; q(\tilde{\phi}_i))$ in (\ref{equ:Immpmq}). Therefore,
\begin{equation*}
	M_n =\sum_{i=1}^{n} \left[ \Delta Z_{i}-J(m,m'; q(\tilde{\phi}_{i})) \right]  = Z_{n} - \sum_{i=1}^{n} J(m,m'; q(\tilde{\phi}_{i}))
\end{equation*}
forms a martingale under $\textbf{P}_m$ with respect to $\{\mathcal{F}_n\}$.  Applying the optional stopping theorem to the martingale $\{M_n; \mathcal{F}_n\},$ for the stopping time $N$ of a decentralized test $\delta(c),$ we have ${\bf E}_{m}(M_{N}) = 0,$ or equivalently,
\begin{equation}\label{equ:Mn}
	\textbf{E}_m\left\{ Z_N \right\}  = \textbf{E}_m\left\{ \sum_{i=1}^{N} J(m,m'; q(\tilde{\phi}_i)) \right\}.
\end{equation}

Now let us go back to the proof of Theorem \ref{the:asympoptdeltaIc}. Obviously, for a decentralized test $\delta(c)$  satisfying the error probability assumption in Theorem \ref{the:asympoptdeltaIc}, if the sample size $N$ satisfies $\textbf{E}_m\left\{ N \right\} = \infty,$ then Theorem \ref{the:asympoptdeltaIc} holds. Thus we only need to consider the case when $\textbf{E}_m(N) < \infty.$ To derive the asymptotic lower bound on $\textbf{E}_m(N),$ we construct a new test $\delta'(c)$ that accepts $\textbf{H}_m$ if the final decision of $\delta(c)$ is $D=m$ but accepts $\textbf{H}_{m'}$ (for a given $m' \ne m$) if $D \ne m.$ Then this new test $\delta'(c)$ is a well-defined sequential test in the problem of testing a simple hypothesis $\textbf{H}_{m}$ against a simple alternative $\textbf{H}_{m'}.$ Moreover, the assumption of Theorem \ref{the:asympoptdeltaIc} guarantees that both type I and type II errors of $\delta'(c)$ are less than $\alpha_{c} = Ac|\log c|,$ where $A>0$ is a constant. Hence, $Z_N$ represents the log-likelihood ratio of the test  $\delta'(c)$ when stopped and
by Wald's inequalities (also see Theorem 2.39 of Siegmund \cite{sie}),
  \begin{eqnarray*}
	  \textbf{E}_m\left\{ Z_N \right\} &\ge& (1-\alpha_{c})\log(\frac{1-\alpha_{c}}{\alpha_{c}})+\alpha_{c}\log(\frac{\alpha_{c}}{1-\alpha_{c}}) \\
&\ge& (1-\alpha_{c}) |\log \alpha_{c}| - \log 2 \\
&=& |\log c|- \log|\log c|+O(1)
  \end{eqnarray*}
as $c \rightarrow 0,$ where the $O(1)$ term depends only on $A$.
Here the second inequality follows from the facts that $\alpha \log (1-\alpha)^{-1}$ is nonnegative and that $(1-\alpha) \log (1-\alpha) + \alpha \log \alpha$ attains minimum value $-\log 2$ when $\alpha = \frac12.$ 
By (\ref{equ:Mn}), we have
\begin{equation}\label{eqnn031}
\textbf{E}_m\left\{ \sum_{i=1}^{N} J(m,m'; q(\tilde{\phi}_{i})) \right\}\ge |\log c|-\log|\log c|+O(1).
 \end{equation}

Now we claim that the left-hand side of (\ref{eqnn031}) can be rewritten as $J^*(m,m'; \mu_m) \textbf{E}_m\left\{ N \right\}$
for a suitably chosen probability measure $\mu_m$ on $\bar{Q},$ where $J^*(m,m' ; \mu_m)$ is defined as in (\ref{equ:ImmprimMu}).
Then the theorem follows at once from this claim, relation (\ref{eqnn031}), and Lemma \ref{lem:ImmuSupIsIm}.
It remains to prove this claim. To do so, define
$\mu_m$ as a convex combination of a sequence of probability measures $\{\mu_{m,n,i}: i\le n\}$ as follows.
 \[\mu_m=\sum_{n=1}^{\infty}\sum_{i=1}^{n}\frac{\textbf{P}_m\{N=n\}}{\textbf{E}_m\{N\}}\mu_{m,n,i}.\]
 Then let $\mu_{m,n,i}$ be the distribution of $q(\tilde{\phi}_i)$ under $\textbf{P}_m$ and conditioned on the event $N=n$. In other words, for any Borel set $A\subset \bar Q$, $\mu_{m,n,i}(A)=\textbf{P}_m\{q(\tilde{\phi}_i)\in A|N=n\}$.  We have
 \begin{eqnarray*}
	 &&\textbf{E}_m\left\{ N \right\}J^*(m, m' ; \mu_m)\\
	 &=& \textbf{E}_m\left\{ N \right\} \sum_{n=1}^{\infty}\sum_{i=1}^{n} \frac{\textbf{P}_m\left\{ N=n \right\}}{\textbf{E}_m\left\{ N \right\}} J^*(m,m'; \mu_{m,n,i})\\
	 &=& \sum_{n=1}^{\infty}\textbf{P}_m\left\{ N=n \right\}\sum_{i=1}^{n}\textbf{E}_m\left\{\left. J(m,m'; q(\tilde{\phi}_i))\right|N=n \right\}\\
	 &=& \sum_{n=1}^{\infty}\sum_{i=1}^{n}\textbf{E}_m\left\{ J(m,m'; q(\tilde{\phi}_i)) 1\{N=n\} \right\}\\
	 &=& \textbf{E}_m\left\{ \sum_{i=1}^{N} J(m,m'; q(\tilde{\phi}_i)) \right\}.
 \end{eqnarray*}


%





\ifCLASSOPTIONcaptionsoff
  \newpage
\fi




\begin{thebibliography}{1}
\bibitem{bv}
  C. W. Baum, V. V. Veeravalli, ``A sequential procedure for multihypothesis testing'', \textit{IEEE Trans. Inf. Theory}, vol. 40, pp. 1994-2007, Nov. 1994.
\bibitem{blum97}
R. S. Blum, S. A. Kassam, and H. V. Poor, ``Distributed detection
with muliple sensors: part II---advanced topics," \emph{Proceedings
of the IEEE}, vol. 85, no. 1, pp. 64-79, Jan. 1997.
\bibitem{che}
H. Chernoff, ``Sequential design of experiment,'' \textit{Ann. Math. Statist.}, vol. 30, pp. 755-770, Sep. 1959.

\bibitem{che72}
H. Chernoff, \textit{Sequential Analysis and Optimal Design.} Philadelphia, PA: SIAM, 1972.
\bibitem{dtv}
V. P. Dragalin, A. G. Tartakovsky, V. V. Veeravalli, ``Sequential Probability Ratio Tests---Part I: Asymptotic Optimality'', \textit{IEEE Trans. Inf. Theory}, vol. 45, pp. 2448-2461, Nov. 1999.

\bibitem{dtv2}
V. P. Dragalin, A. G. Tartakovsky, V. V. Veeravalli, ``Sequential Probability Ratio Tests---Part II: Accurate Asymptotic Expansions for the Expected Sample Size'', \textit{IEEE Trans. Inf. Theory}, vol. 46, pp. 1366-1383, Jul. 2000.

\bibitem{fu}
  K. S. Fu, \textit{Sequential Methods in Pattern Recognition and Learning,} New York: Academic, 1968.

  \bibitem{gut}
	  A. Gut, \textit{Stopped Random Walks: Limit Theorems and Applications.} New York: Springer-Verlag, 1988.
\bibitem{hor}
L. Hormander, \textit{Notions Of Convexity}, Chapter II.1. Boston: Birkhauser, 1994.

\bibitem{ks}
J. Kiefer and J. Sacks, ``Asymptotically optimal sequential inference and design,'' \textit{Ann. Math. Statist.}, vol. 34, pp. 705-750, Sep. 1963.

\bibitem{klass} M. J. Klass, ``On the maximum of a random walk with small negative drift,'' \textit{Ann. Probab.,} vol. 11, No. 3, pp. 491-505, 1983.


\bibitem{leh}
  E. L. Lehmann, \textit{Testing Statistical Hypotheses.} New York: Wiley, 1959.

\bibitem{lwhs} D. Li, K. D. Wong, Y. H. Hu, A. M. Sayeed, ``Detection, classification and tracking of targets in distributed sensor networks'', \textit{IEEE Signal Processing Magazine,} vol. 19, pp 17-29, Mar. 2002.

\bibitem{lord} G. Lorden, ``Nearly-optimal sequential tests for finitely many parameter values,'' \textit{Ann. Statist.,} vol. 5, No.1, pp. 1-21, 1977.

\bibitem{ms} M. B. Marcus and P. Swerling, ``Sequential detection in radar with multiple resolution elements,'' \textit{IRE Trans. Inform. Theory,} pp. 237-245, Apr. 1962.

\bibitem{mei05}
  Y. Mei, ``Information bounds and quickest change detection in decentralized decision systems,'' \textit{IEEE Trans. Inf. Theory, } vol. 51, pp. 2669-2681, Jul. 2005.
\bibitem{mei}
Y. Mei, ``Asymptotic optimality theory for decentralized sequential hypothesis testing in sensor networks'' \textit{IEEE Trans. Inf. Theory}, vol. 54, pp. 2072-2089, May. 2008.
\bibitem{nw} X. Nguyen, M. J. Wainwright, M. I. Jordan, ``On optimal quantization rules for some problems in sequential decentralized detection'' \textit{IEEE Trans. Inf. Theory}, vol. 54, pp. 3285-3295, Jul. 2008.

\bibitem{sosl} M. K. Simon, J. K. Omura, R. A. Scholtz, and B. K. Levitt, \textit{Spread Spectrum Communications,} vol. III. Rockville, MD: Comput. Sci., 1985.


\bibitem{stein}
C. Stein, ``A two-sample test for a linear hypothesis whose power is independent of the variance,'' \textit{Ann. Math. Statist.,} vol. 16, No. 3, pp. 243-258, 1945.

\bibitem{ts}
  R. R. Tenney, N. R. Sandell Jr., ``Detection with distributed sensors,'' \textit{IEEE Trans. Aerospace Elect. Syst., } vol. AES-17, pp.501-510, Jul. 1981.
\bibitem{tsito}
  I. I. Tsitovich, ``Sequential design of experiments for hypothesis testing'' \textit{Theory Prob. Appl.,} vol. 29, pp. 814-817, Jan. 1985.
\bibitem{tsi}
J. N. Tsitsiklis, ``Extremal properties of likelihood ratio quantizers'', \textit{IEEE Trans. Commun.}, vol. 41, pp. 550-558, Apr. 1993.
\bibitem{sie}
D. Siegmund, \textit{Sequential Analysis, Tests and Confidence Intervals,} New York: Springer-Verlag, 1985.
\bibitem{van}
H. L. Van Trees, \textit{Detection Estimation and Modulation Theory,} vol. I. New York: Wiley, 1968.
\bibitem{vee}
V. V. Veeravalli, ``Sequential decision fusion: theory and applications'', \textit{J. Franklin Inst.}, vol. 336, pp. 301-322, Feb. 1999.
\bibitem{vbp}
V. V. Veeravalli, T. Basar, and H. V. Poor, ``Decentralized sequential detection with a fusion center performing the sequential test,'' \textit{IEEE Trans. Inf. Theory}, vol. 39, pp. 433-442, Mar. 1993.
\bibitem{vv}
  R. Viswannathan, P. K. Varshney, ``Distributed detection with muliple sensors: part I- Fundamentals," \emph{Proceedings of the IEEE}, vol. 85, no. 1, pp. 54-63, 1997.
\bibitem{wald47} A. Wald, \emph{Sequential Analysis}. New York: Wiley, 1947.
\bibitem{ww}
  A. Wald and J. Wolfowitz, ``Optimal character of the sequential probability ratio test,'' \textit{Ann. Math. Statist.,} vol. 19, pp. 326-339, Sep. 1948.
\bibitem{yllz} F. Ye, H. Luo, S. Lu, L. Zhang, ``Statistical en-route filtering of injected false data in sensor networks'', \textit{IEEE Journal on Selected Areas in Communications}, vol 23, pp 839-850, Apr. 2005.

\end{thebibliography}
%

\end{document}